\documentclass[11pt]{amsart}
 \usepackage[utf8]{inputenc}
\usepackage{color,amssymb, thm-restate,hyperref}
\usepackage[normalem]{ulem}
\usepackage[shortlabels]{enumitem}

\usepackage[capitalise]{cleveref}
\usepackage{mathrsfs}

\usepackage{ocgx2}
\usepackage{graphicx}

\usepackage[all]{xy}

\newtheorem{theorem}{Theorem}[section]
\newtheorem{claim}[theorem]{Claim}

\newtheorem{corollary}[theorem]{Corollary}

\newtheorem{lemma}[theorem]{Lemma}

\newtheorem{question}[theorem]{Question}

\newtheorem*{claim*}{Claim}

\theoremstyle{definition}
\newtheorem{remark}[theorem]{Remark}
\newtheorem{definition}[theorem]{Definition}

\newtheorem*{definition*}{Definition}

\newcommand{\Itwo}{\text{\(\mathcal{I}_2\)}}
\newcommand{\Ithree}{\text{\(\mathcal{I}_3\)}}
\newcommand{\Ifour}{\text{\(\mathcal{I}_4\)}}

\newcommand{\N}{\mathbb N}

\newcommand{\F}{\mathcal F}

\newcommand{\w}{\omega}

\newcommand{\Tau}{\mathcal T}

\newcommand{\im}{\operatorname{im}}
\newcommand{\dom}{\operatorname{dom}}
\newcommand{\Sym}{\operatorname{Sym}}

\def\id{\operatorname{id}}

\newcommand{\set}[2]{\{#1:#2\}}

\newcommand{\makeset}[2]{\left\lbrace #1 \;\middle|\;
 \begin{tabular}{@{}l@{}}
   #2
  \end{tabular}
  \right\rbrace}

\author{S. Bardyla, L. Elliott, J. D. Mitchell, and Y. P\'eresse}

\address{S.~Bardyla: University of Vienna, Institute of Mathematics, Vienna, Austria}
\thanks{The research of the first named author was funded in whole by the Austrian Science Fund FWF [10.55776/ESP399].}
\email{sbardyla@gmail.com}
\address{L. Elliott: University of Binghamton, Department of Mathematics and Statistics, Binghamton, SUNY, USA}
\email{luna.elliott142857@gmail.com}
\address{J. D. Mitchell: University of St Andrews, School of Mathematics and Statistics, Scotland, UK}
\email{jdm3@st-andrews.ac.uk}
\address{Y. P\'eresse: University of Hertfordshire, School of Physics, Engineering and Computer Science, UK}
\email{y.peresse@herts.ac.uk}
\makeatletter

\subjclass[2020]{20M18, 20M20, 22A15, 54H15}
\keywords{Symmetric inverse monoid, Baire space, Polish semigroup, poset of Polish topologies}

\title{Classifying the Polish semigroup topologies on the symmetric inverse monoid}

\begin{document}
    
    \begin{abstract}
    We classify all Polish semigroup topologies on the symmetric inverse monoid $I_\N$ on the natural numbers $\N$. This result answers a question of Elliott et al. There are countably infinitely many such topologies. Under containment, these Polish semigroup topologies form a join-semilattice with infinite descending chains, no infinite ascending chains, and arbitrarily large finite anti-chains. Also, we show that the monoid $I_\N$ endowed with any second countable $T_1$ semigroup topology is homeomorphic to the Baire space $\N^\N$.   
    \end{abstract}
\maketitle

\section{Introduction}

Recall that the \emph{full transformation monoid} $X ^ X$ on a set $X$ consists of all functions from $X$ to $X$ under the operation of composition of functions. The monoid $X^X$ plays the same role
in semigroup theory, as the symmetric group $S_X$ does for groups. That is, by an analogue of Cayley's Theorem
every semigroup embeds into some full transformation monoid. 

From a set-theoretic point of view, $X^X$ is just the Cartesian product of $|X|$ copies of $X$. Thus $X^X$ has a natural topology, namely the (Tychonoff) product topology arising from the discrete topology on each copy of $X$. This product topology, also known as the \emph{pointwise topology}, is compatible with the algebraic structure of $X ^ X$ in the following sense. The semigroup multiplication on $X^X$, seen as a function from $X
^X \times X^X$ to $X^X$, is continuous under the pointwise topology. This makes $X ^ X$, equipped with the pointwise topology, a \emph{topological semigroup}. The symmetric group $S_X$ is a subspace of $X^X$ and forms a \emph{topological group} under the subspace topology. That is, the multiplication map $(f,g) \mapsto fg$ and the inversion map $g\mapsto g^{-1}$ are continuous under the pointwise topology on $S_X$.

The pointwise topology on $X^X$ is of particular interest in the case when $X$ is countably infinite, say $X=\N=\{0,1,2,\dots\}$. In this case, $X ^ X=\N^\N$ is the so-called Baire space. The Baire space is a Polish
space, i.e. a separable completely metrizable topological space. The symmetric group $S_{\N}$ is a $G_{\delta}$ subspace of $\N ^ \N$, and hence is a Polish space also. Polish spaces and groups are widely studied in descriptive set theory, see~\cite{Clas, YBM, CK, main1, GP, K, K1, K2, KM, Kechris, Mon, R}. 
The intrinsicness between the topology of the Baire space and the algebraic structures of the monoid $\N ^
{\N}$ and the group $S_{\N}$ goes further.
The Baire space topology is
the unique Polish semigroup topology on $\N ^ \N$~\cite[Theorem 5.4]{main}; and, similarly, the subspace topology inherited from the Baire space is the 
unique Polish group topology on $S_{\N}$ \cite{Gaughan, KR, RS}. For more about Polish topologies on groups and semigroups and their relation to automatic continuity see~\cite{our, BP, BE,  D, main, HHLS, L, M, Per, PS, PS2, Sa, Tr, T,Y}.

Inverse semigroups are semigroups
satisfying a ``local inverses'' axiom: for every $x\in S$ there exists a unique $x^{-1}\in S$ such that $xx^{-1}x = x$ and $x^{-1}xx^{-1} = x^{-1}$. Being a natural generalization of groups and semilattices, inverse semigroups have been widely studied in the literature; see the monographs~\cite{howie1995fundamentals, Lawson1998, Petrich1984} and references therein. An inverse semigroup $S$ equipped with a topology $\tau$ is called a {\em topological inverse semigroup} if both multiplication and inversion are continuous in $(S,\tau)$. 
The {\em symmetric inverse monoid} $I_X$ consists of all bijections between subsets of $X$ under the usual composition of binary relations.
By the Wagner-Preston Theorem~\cite[Theorem 5.1.7]{howie1995fundamentals}, $I_X$ plays the same role for the class of inverse semigroups as $X ^ X$ for semigroups, or $S_X$ for groups. 
It was shown in \cite[Theorem 5.15]{main} that, analogously to $\N^\N$ and $S_\N$, the symmetric inverse monoid $I_{\N}$ possesses a unique Polish inverse semigroup topology, denoted $\Ifour$ in \cite{main}. The reason for the subscript in the name is that $\Ifour$ is not the only Polish semigroup topology on $I_{\N}$. Whereas a classical result by Montgomery \cite{Mon} implies that every Polish semigroup topology on a group is a group topology, the analogous statement is not true for inverse semigroups. The authors of \cite{main} found two further Polish semigroup topologies $\Itwo$ and $\Ithree$ on $I_\N$. The topologies $\Itwo, \Ithree$ and $\Ifour$ are defined as follows:

For every $x, y\in \N$ let,
\begin{equation*}
    U_{x, y}= \set{h\in I_{\N}}{(x,y)\in h},\quad\hbox{ } W_x = \set{h\in I_{\N}}{ x\notin
    \dom(h)},\quad\hbox{ } W_x^ {-1} =\set{h\in I_{\N}}{x\notin \im(h)}.
\end{equation*}
Then a subbasis for $\Itwo$ is 
\begin{equation}\label{definition_Itwo}
\set{U_{x,y}}{x,y\in \N}\cup\{W_x:x\in\N\};
\end{equation}
a subbasis for $\Ithree$ is
$$\set{U_{x,y}}{x,y\in \N}\cup\{W^{-1}_x:x\in\N\};$$
and $\Ifour$ is the topology generated by $\Itwo \cup \Ithree$.

It was also shown in \cite{main} that every Polish semigroup topology for
$I_{\N}$ is contained in $\Ifour$ and contains either $\Itwo$ or $\Ithree$ which
led the authors to ask the following question:

\begin{question}[{\cite[Question 5.17]{main}}]\label{q2}
Are $\Itwo$, $\Ithree$, and $\Ifour$ the only Polish semigroup topologies on $I_{\N}$?
\end{question}

In this paper, we answer \cref{q2} by classifying all Polish semigroup
topologies on the symmetric inverse monoid $I_\N$; see \cref{tau}, \cref{definition-waning},
and \cref{main-theorem}. In
particular, there are countably many such semigroup topologies, forming a
join-semilattice with: infinite descending chains, no infinite ascending chains,
and arbitrarily long finite anti-chains; see \cref{corollary-order}. 
If $\omega$
denotes the first infinite ordinal, then each such Polish semigroup topology is
characterised by a special kind of function $f:\omega +
1\to \omega+1$ that we will refer to as \emph{waning}; see
\cref{definition-waning}. The order on the join-semilattice of the Polish
semigroup topologies on $I_\N$ is also completely characterised in terms of the
corresponding waning functions, see \cref{newversion}.
Even though $I_{\N}$ can be made into a Polish semigroup in infinitely many different ways, we always obtain the Baire space.
More precisely, we show that the symmetric inverse monoid $I_\N$ endowed with any \(T_1\) second countable semigroup topology is homeomorphic to the Baire space $\N^\N$, see~\cref{newone}.

This paper is organised as follows: in \cref{section2} we provide the necessary
definitions, and then state the main theorems. In
\cref{section-waning-func-to-topology} we show how to associate a Polish
semigroup topology on $I_{\N}$ to every waning function and, conversely, in
\cref{section-from-topologies-to-waning} we show that every Polish semigroup
topology on $I_{\N}$ equals one of the topologies from
\cref{section-waning-func-to-topology}. In \cref{section-corollaries} we prove the remaining main results. 

\section{Main results}\label{section2}

Following usual conventions in Set Theory, we identify the set of natural numbers $\N$ with the first infinite ordinal $\omega=\{0,1,2, \dots\}$ and a natural number $n$ with the ordinal $\{0, \dots, n-1\}$.
The topologies on $I_{\N}$ arising from functions from the successor ordinal $\omega + 1$ of $\omega$ to itself are defined as follows.
\begin{definition}[\textbf{The topology $\mathcal{T}_f$}] \label{tau}
Suppose that \(f:\omega + 1 \to \omega + 1\) is any function. We define
$\mathcal{T}_f$ to be the least topology on $I_{\N}$ containing $\Itwo$ and the
sets of the form:
\[
U_{f, n, X}= \makeset{g\in I_\N}{\(|\im(g) \setminus X|\geq n\) and \(|X\cap \im(g)|\leq (n)f\)}\]
where \(n\in \N\) and \(X\subset \N\) is finite.   
\end{definition}

Roughly speaking, the open sets \(U_{f, n, X}\) consist of those elements of $I_{\N}$ which try not to map any points into \(X\). 
An element \(g\in U_{f, n, X}\) must map at least $n$ points outside of $X$ and is allowed at most $(n)f$ mistakes, i.e. points in the image of $g$ in $X$. 

It will transpire that every function $f: \omega + 1 \to \omega + 1$ gives rise to a Polish semigroup topology $\Tau_f$ on $I_{\N}$ (\cref{theorem-waning-function-to-Polish-topology}). However, 
distinct functions can give rise to the same topology. To circumvent this, we require the following definition.

\begin{definition}[\textbf{Waning function}]\label{definition-waning}
  We say that a non-increasing function \(f:\omega + 1\to \omega + 1\) is
  \textit{waning}  if either: $f$ is constant with value $\omega$; or there exists
  \(i \in \omega\) such that $(i)f \in \omega$, and \((j+1)f< (j)f\) for all $j
  \in \omega$ such that  \(0\neq (j)f \in \omega\).
\end{definition}

For example, the constant function $\overline{0}$ with value $0$ is a waning function, 
and so too is the following function:
\begin{equation*}
(x)f = 
\begin{cases}
\omega & \text{if } x \leq 42\\
1337 - x & \text{if } 42 \leq x < 69 \\
0 & \text{if } x \geq 69.
\end{cases}
\end{equation*}

 If $f$ is a waning function, continuing to speak roughly, partial functions in
 $U_{f, n, X}$ defined on more points are allowed fewer mistakes. 

For a topology $\Tau$ on $I_\N$ let $\Tau^{-1}=\{U^{-1}:U\in\Tau\}$. It is
routine to verify that if $\Tau$ is a (Polish) semigroup topology on $I_\N$, then so is
$\Tau^{-1}$. The following result answers \cref{q2} and is the main result of this paper.

\begin{theorem}\label{main-theorem}
  If $\Tau$ is a \(T_1\) second countable semigroup topology on \(I_\N\), then
  there exists a waning function $f$ such that either $\Tau=\Tau_f$, or
  $\Tau^{-1}=\Tau_f$. 
  
  Conversely, if $f$ and $g$ are distinct waning functions, then 
  $\Tau_f$ and $\Tau_g$ are distinct Polish semigroup topologies for $I_{\N}$.
\end{theorem}

Clearly, there are only countably many waning functions, and so \cref{main-theorem} has the
following corollary.

\begin{corollary}
The symmetric inverse monoid has countably infinitely many distinct Polish semigroup topologies.
\end{corollary}

The remaining main results are less immediate consequences of \cref{newversion} and will be proved in \cref{section-corollaries}.
As noted earlier, every $T_1$ second-countable topology for $I_{\N}$ is Polish. In fact, up to homeomorphism, there is only one such topology.

\begin{theorem}\label{newone}
The symmetric inverse monoid $I_\N$ endowed with any \(T_1\) second countable semigroup topology is homeomorphic to the Baire space $\N^\N$.
\end{theorem}

As discussed in the introduction, we now address the question of how the Polish semigroup topologies on $I_{\N}$ are ordered with respect to containment.  
If $f, g
\in (\omega + 1)^{\omega+1}$, we write $f \preceq g$ whenever $(i)f \geq (i)g$
for all $i\in \omega + 1$.  The reason that we order $f\preceq g$ when $f$ is
coordinatewise greater than $g$ will become apparent shortly. We denote by
$\mathfrak{W}$ the set of all waning functions on $\omega + 1$ with the partial
order $\preceq$.


Recall that, by \cite{main}, every $T_1$ second-countable semigroup topology on $I_{\N}$ contains $\Itwo$ or $\Ithree=\Itwo^{-1}$ and is contained in $\Ifour=\Ifour^{-1}$. It follows that the entire poset $\mathfrak{P}$ of Polish semigroup topologies on $I_{\N}$, with respect to containment, is described by the following theorem. 

\begin{theorem}\label{newversion}
{\em (i)} The function $f \mapsto \mathcal{T}_f$ is an order-isomorphism from $\mathfrak{W}$
to the subposet of $\mathfrak{P}$ consisting of topologies which contain $\Itwo$ and are contained in $\Ifour$.

{\em (ii)} the function $f \mapsto \mathcal{T}_f ^ {-1}$ is an order-isomorphism from
$\mathfrak{W}$ to the subposet of $\mathfrak{P}$ of topologies lying between $\Ithree$ and $\Ifour$.
\end{theorem}

The reason for defining $\preceq$ as we did, should now be clear: to avoid
having an ``anti-embedding'' and ``anti-isomorphism'' in the statement of
\cref{newversion}. 

\begin{theorem}\label{corollary-order}
  The partial order of Polish semigroup topologies on $I_{\N}$ contains the following:
  \begin{enumerate}[\rm (a)]
  \item infinite descending linear orders;
  \item only finite ascending linear orders; and 
  \item every finite partial order.
  \end{enumerate}
\end{theorem}


\section{From waning functions to semigroup topologies}\label{section-waning-func-to-topology}

In this section we will prove the following for any waning function $f: \omega + 1 \to \omega + 1$:
\begin{enumerate}[\rm (a)]
   \item in \cref{neighbourhood_basis} we establish a convenient neighbourhood basis for $\mathcal{T}_f$ at every $g\in I_{\N}$; 
    \item  in \cref{theorem-waning-function-to-Polish-topology} we show that 
    $\mathcal{T}_f$ (defined in \cref{definition-waning}) is a Polish
    semigroup topology on the symmetric inverse monoid $I_{\N}$;
   \item 
   in \cref{lem-order-preserving} we show that if $f,h:\omega \rightarrow \omega$ are waning functions, then $\Tau_f \subseteq \Tau_h$ if and only if $f\preceq
   h$.
\end{enumerate}
In terms of proving the main results of the paper, (b) above proves one direction of Theorem~\ref{main-theorem}, and (c) together with Theorem~\ref{main-theorem} proves Theorem~\ref{newversion}.


%


 Let \(g\in I_\N\). We denote by $|g|$ the size of $g$ as a subset of $\N \times \N$ (i.e. $|g| = |\set{(x, (x)g)}{x\in \N}|$), and by $g{\restriction}_r:=\makeset{(x, y)\in g}{\(x\in r\)}$ the restriction of $g$ to $\{0, \ldots, r - 1\}$.

\begin{lemma}\label{lemma-new}
If $f$ is a waning function and $g\in I_{\N}$, then there exists $r\in \N$ such that
$(r)f \leq  (|g|)f=(|g{\restriction}_r|)f$.
\end{lemma}
\begin{proof}
If $f$ is constant with value $\omega$, then we define $r = 0$ and  so
$(r)f = \omega \leq \omega = (|g|)f=(|g{\restriction}_r|)f$.

Suppose that $f$ is not a constant function with value $\omega$. If $|g| < \omega$, then we define $r\in \N$ so that 
$r> \max (\dom(g))$. In this case, $g{\restriction}_r = g$ and $(r)f\leq (|g|)f$ since $|g| < r$. This implies that  $(r)f\leq (|g|)f = (|g{\restriction}_r|)f$. 

If  $|g| = \omega$, then we choose $r$ large enough so that $(|g{\restriction}_{r}|)f = 0$. Since $r \geq |g{\restriction}_{r}|$, it follows that $(r)f \leq (|g{\restriction}_{r}|)f = 0 = (\omega)f = (|g|)f$.
\end{proof} 

For any waning function $f$, any $g\in I_{\N}$, and any $r\in
\N$ such that the condition in \cref{lemma-new} holds we define:
\begin{equation}
  \label{definition}
  W_{f,g, r}=\{h\in I_\N: h{\restriction}_{r}=g{\restriction}_r \text{  and  }
  |\im(h) \cap (r\setminus \im(g)) |\leq (|g|)f\}.
\end{equation}

\begin{lemma}\label{explanation}
    If $f$ is a waning function, $g\in I_{\N}$, and $r\in \N$ are such that $W_{f,
    g, r}$ is defined, then for any $p\geq r$, $W_{f, g, p}$ is defined and $W_{f,
    g, p} \subseteq W_{f, g, r}$.    
\end{lemma}
\begin{proof}
Assume that $W_{f, g, r}$ is defined, i.e. $(r)f \leq (|g|)f=(|g{\restriction}_r|)f$. Fix any $p\geq r$. Since $f$ is a waning function and by the choice of $r$ we have the following: 
$$(p)f\leq (r)f\leq (|g|)f=(|g{\restriction}_r|)f\geq (|g{\restriction}_p|)f \geq (|g|)f.$$ 
It follows that $(p)f \leq (|g|)f=(|g{\restriction}_p|)f$ and, consequently, the set $W_{f, g, p}$ is defined. The inclusion $W_{f, g, p} \subseteq W_{f, g, r}$ follows easily from $p\geq r$. 
\end{proof}



\begin{definition}\label{max_smaller_wan}
For all \(f\in (\omega+1)^{(\omega+1)}\)  we define \(f':\omega+1 \to \omega+1\) to be the unique waning function obtained inductively as follows:
\begin{itemize}
    \item Let $(0)f'=(0)f$.
    \item  If $i\in \omega$ and $(i)f'\neq 0$, then let $(i+1)f'=\min\{ (i+1)f, (i)f'-1\}$.
    \item If $i\in \omega$ and $(i)f'=0$, then let $(i+1)f'=0$.
    \item if $(i)f'=\omega$ for all $i \in \omega$, then let $(\omega)f'=\omega$, and otherwise let $(\omega)f'=0$. 
\end{itemize}
In other words, for every \(i\in \omega\) we have
\[(i)f' = 
   \max\{0,\min\{(j)f-(i-j): j\leq i\}\},
\]
where we use the convention that $\omega-n=\omega$ for every $n\in \omega$. Note that $f'$ is equal to the minimum waning function with respect to the order $\preceq$ on $(\w+1)^{(\w+1)}$ such that \((i)f' \leq (i)f \) for all \(i\in \omega\). In particular, $f'=f$ whenever $f$ is a waning function. 
\end{definition}

\begin{lemma}\label{much_wan}
Let \(f\in (\omega+1)^{(\omega+1)}\) and $g\in I_\N$. If \(r\in\N\) satisfies \((r)f' \leq (|g|)f'=(|g{\restriction}_r|)f'\), then the set \(W_{f',g, r}\) is an open neighborhood of $g$ in \((I_\N,\Tau_f)\).
\end{lemma}
\begin{proof}



By \cref{max_smaller_wan}, there is some \(j\leq|g{\restriction}_r|\) such that
\((|g{\restriction}_{r}|)f'\geq (j)f-(|g{\restriction}_{r}|-j)\). Put
\[b= (|g{\restriction}_r|)f'+|g{\restriction}_r|-(j)f.\]
Observe that if \((|g{\restriction}_{r}|)f'\neq 0\), then \((|g{\restriction}_{r}|)f'= (j)f-(|g{\restriction}_{r}|-j)\) and hence $b=j$, and if \((|g{\restriction}_{r}|)f'=0\), then  \(b= |g{\restriction}_r|-(j)f\geq j\). Note that in either case we have the following:
\begin{enumerate}
    \item[(i)] \(j\leq b\leq |g{\restriction}_r|\).
    \item[(ii)] \((|g{\restriction}_r|)f'=(j)f-(|g{\restriction}_r|-b)\).
\end{enumerate}

Let \(Y= r\setminus \im(g)\) and \(Z\subseteq \im(g{\restriction}_{r})\) be such that \(|Z|=|g{\restriction}_{r}|-b\). Then
\begin{equation}\label{1}
  (|g{\restriction}_{r}|)f'=(j)f-|Z|.  
\end{equation}
Let \(X=Y\cup Z\) and note that \(b=|\im(g{\restriction}_r) \setminus Z|=|g{\restriction}_r|-|Z|\). 
Then $\im(g{\restriction}_r) \setminus X =\im(g{\restriction}_r) \setminus Z$ and so 
\begin{equation}\label{2}
|\im(g) \setminus X| \geq|\im(g{\restriction}_r) \setminus X|= |\im(g{\restriction}_r) \setminus Z|= b\geq j,
\end{equation}
and $X\cap \im(g) = Z\cap \im (g)$ and so 
\[|X \cap \im(g) | = |Z \cap \im(g)| = |Z| = |g{\restriction}_r| - b =(j)f-(|g{\restriction}_r|)f'\leq  (j)f.\]
So \(g\in U_{f, j, X}=\makeset{l\in I_\N}{\(|\im(l) \setminus X|\geq j\) and \(|X\cap \im(l)|\leq (j)f\)}\). Moreover
\begin{align*}
    g&\in \makeset{l\in I_\N}{\(l{\restriction}_r= g{\restriction}_r\)}\cap U_{f, j, X}\\
    &= \makeset{l\in I_\N}{\(l{\restriction}_r= g{\restriction}_r\) and \(|\im(l) \setminus X|\geq j\) and \(|X\cap \im(l)|\leq (j)f\)}\\
      &= \makeset{l\in I_\N}{\(l{\restriction}_r= g{\restriction}_r\) and \(|X\cap \im(l)|\leq (j)f\)}  \quad\quad &&(\text{by \ref{2}})\\
    &= \makeset{l\in I_\N}{\(l{\restriction}_r= g{\restriction}_r\)  and \(|\im(l) \cap Z|+|\im(l)\cap Y|\leq (j)f\)}\quad\quad &&(\text{as } X=Z \sqcup Y)\\
        &= \makeset{l\in I_\N}{\(l{\restriction}_r= g{\restriction}_r\)  and \(|Z|+|\im(l)\cap Y|\leq (j)f\)} 
    \quad\quad &&(\text{as } Z\subseteq \im(g{\restriction}_r)=\im(l{\restriction}_r)) \\
      &= \makeset{l\in I_\N}{\(l{\restriction}_r= g{\restriction}_r\)  and \(|\im(l)\cap Y|\leq (j)f-|Z|\)}\\
       &= \makeset{l\in I_\N}{\(l{\restriction}_r= g{\restriction}_r\)  and \(|\im(l)\cap Y|\leq (|g{\restriction}_r|)f'\)}\quad\quad&&(\text{by } \ref{1})\\
       &= \makeset{l\in I_\N}{\(l{\restriction}_r= g{\restriction}_r\) and  \(|\im(l)\cap Y|\leq  (|g|)f'\)}\\
       &=\makeset{l\in I_\N}{\(l{\restriction}_r= g{\restriction}_r\) and \(|\im(l)\cap (r\setminus \im(g))|\leq  (|g|)f'\)}= W_{f', g, r}. 
       && \qedhere
\end{align*}
\end{proof}

\begin{lemma}\label{neighbourhood_basis}
If \(f\) is a waning function and \(g\in I_\N\),
then the family \[
\mathcal{B}_{f}(g) = \{W_{f, g, r}: (r)f \leq (|g|)f=(|g{\restriction}_r|)f\}\] forms an open neighbourhood basis at \(g\) in \((I_\N,\Tau_f)\).
\end{lemma}
\begin{proof}
Clearly, for each open neighbourhood $V$ of $g$ in $(I_\N,\Itwo)$, there exists large enough $r\in\N$ such that the set $W_{f,g,r}$ is open in $(I_\N,\Tau_f)$ (see \cref{much_wan}) and $g\in W_{f,g,r}\subseteq V$. Let $n\in\N$ and $X$ be a finite subset of $\N$ such that 
\[g\in U_{f, n, X}=\{h\in I_\N: |\im(h)\setminus X|\geq n\hbox{ and }|X\cap \im(h)|\leq (n)f\}.\]

Let \(r\in \N\) be larger than all the elements of \(X\), and large enough that \(|\im(g{\restriction}_r)\setminus X|\geq n\) and $(r)f\leq (|g|)f=(|g{\restriction}_r|)f$. By \cref{much_wan}, $g\in W_{f,g,r}$ and $W_{f,g,r}\in \Tau_f$. 
 We show that $W_{f,g,r}\subseteq U_{f, n, X}$.

Let \(h\in W_{f, g, r}\). As \(h{\restriction}_r=g{\restriction}_r\) we have \(|\im(h)\setminus X|\geq n\). 
It remains to show that \(|X\cap \im(h)|\leq (n)f\). If $(n)f=\w$, then there is nothing to show. Assume that $(n)f\in\w$. By the assumption we have that
\[|\im(h)\cap (r\setminus \im(g))|\leq (|g|)f.\]
So, \(|\im(h)\cap (X\setminus \im(g))|\leq (|g|)f\) and hence
\begin{align*}
    |\im(h)\cap X|&=|\im(h)\cap\im(g)\cap X |+|\im(h)\cap (X\setminus \im(g))|\\
    &\leq|\im(g)\cap X |+|\im(h)\cap (r\setminus \im(g))|\\
    &\leq|\im(g)\cap X |+ (|g|)f\\
    &\leq|\im(g)\cap X |+ (n + |\im(g)\cap X |)f\quad\quad \text{(as }|\im(g)\setminus X |\geq n).
    \end{align*}
If $(n + |\im(g)\cap X |)f=0$, then   $|\im(h)\cap X|\leq |\im(g)\cap X |\leq (n)f$, as $g\in U_{f,n,X}$. If $(n + |\im(g)\cap X |)f>0$, then taking into account that $f$ is a waning function, we get that 
 $$(n + |\im(g)\cap X |)f\leq (n)f-|\im(g)\cap X|.$$ In the latter case $|\im(h)\cap X|\leq |\im(g)\cap X |+ (n)f-|\im(g)\cap X|=(n)f$. Hence \(|X\cap \im(h)|\leq (n)f\), witnessing that $W_{f,g,r}\subseteq U_{f,n,X}$.
\end{proof}

  Lemmas \ref{much_wan} and \ref{neighbourhood_basis} imply the following.
\begin{corollary}\label{non-wan_vs_wan_1}
If \(f\in (\omega+1)^{(\omega+1)}\) is a function, then \(\mathcal{T}_{f'}\subseteq \mathcal{T}_{f}\).
\end{corollary}

Recall that a set $U$ is a \emph{neighbourhood} of a point $x$ if $x$ lies in the interior of $U$. In particular, $U$ need not be open.
\cref{neighbourhood_basis} implies the following, which will be important in \cref{section-from-topologies-to-waning}.

\begin{remark}\label{remarknew}
For any waning function $f$ the following assertions hold:
\begin{itemize}
    \item For each $h\in I_\N$ such that $(|h|)f=0$, a set $U \subseteq I_{\N}$ is a neighbourhood of $h$ in $\Tau_f$ if and only if $U$ is a neighbourhood of $h$ in $\Ifour$.
    
    \item
    For each $h\in I_\N$ such that $(|h|)f=\w$, a set $U \subseteq I_{\N}$ is a neighbourhood of $h$ in $\Tau_f$ if and only if $U$ is a neighbourhood of $h$ in $\Itwo$.
\end{itemize}
\end{remark}


We will now show that the topologies $\mathcal{T}_f$ are Polish semigroup topologies for $I_{\N}$. We first show that the topology generated by $\Itwo$ and any single set $U_{f, n, X}$ is Polish.

\begin{lemma}\label{newlabel1}
If $f\in (\omega + 1)^{\omega + 1}$, $n\in \N$ and $X$ is a finite subset of $\N$, then the topology $\Tau_{n,X}$ generated by $\Itwo\cup \{U_{f,n, X}\}$ is Polish and possesses a base consisting of clopen sets.
\end{lemma}
\begin{proof}
It is easy to see that the set 
\(\makeset{g\in I_\N}{\(|\im(g) \setminus X|\geq n\)}\)
is open in \((I_\N,\Itwo)\). Note that the subbasis for $\Itwo$ given in \eqref{definition_Itwo} consists of clopen sets and so $\Itwo$ possesses a basis of clopen sets. Thus we can find a family of clopen (with respect to $\Itwo$) sets \(\makeset{U_m}{\(m\in \N\)}\subseteq \Itwo\) such that 
\[\bigcup_{m\in \N}U_m=\makeset{g\in I_\N}{\(|\im(g) \setminus X|\geq n\)}.\]
For each \(m\in \N\), let \(U_{f, n, X, m}= U_{f, n, X}\cap U_m\). By definition, these sets belong to \(\mathcal{T}_{n,X}\). Moreover, 
\[U_{f, n, X}= \bigcup_{m\in \N} U_{f, n, X, m}.\]
So, the topology \(\mathcal{T}_{n,X}\) is equal to the topology generated by \(\Itwo\cup \{U_{f, n, X, m}: m\in\N\}\).

Let us show that for each $m\in\N$ the set $U_{f,n,X,m}$ is closed in $(I_\N,\Itwo)$.  Fix an arbitrary $m\in\N$ and $g\in I_\N\setminus U_{f, n, X,m}$.
  If \(g\not\in U_m\) then we define 
 \[W_g= I_\N \setminus U_m.\]
 Since the set $U_m$ is closed in $(I_\N,\Itwo)$, \(W_g\) is an open neighbourhood of \(g\) in \((I_\N,\Itwo)\) which is disjoint from \(U_{f, n, X, m}\).  If \(g\in U_m\setminus U_{f, n, X,m}\), then \(|\im(g) \setminus X|\geq n\) and \(|X\cap \im(g)|>(n)f\). We define
 \[W_g=\{h\in I_\N\vert  h{\restriction}_{(\im(g)\cap X)g^{-1}}= g{\restriction}_{(\im(g)\cap X)g^{-1}}\}.\]
 Again, \(W_g\) is an open neighbourhood of \(g\) in \((I_\N, \Itwo)\) which is disjoint from \(U_{f, n, X,m}\). So the set $U_{f, n, X,m}$ is closed in $(I_\N,\Itwo)$ as required.

 Since $\Itwo$ is Polish, Lemma 13.2 from~\cite{Kechris} implies that for each
 $m\in\N$  the topology generated by  $\Itwo\cup \{U_{f,n, X,m}\}$ is
 Polish. Lemma 13.3
 from~\cite{Kechris} then implies that the topology \(\mathcal{T}_{n,X}\)
 is Polish as well.
Note that for a given basis $\mathcal B$ of $\Itwo$ consisting of clopen sets, the family $\Theta=\mathcal B\cup \{U_{f, n, X, m}: m\in\N\}$ is a subbasis for \(\Tau_{n, X}\) consisting of clopen sets. Then the finite intersections of elements of $\Theta$ form a basis of  \(\mathcal{T}_{n,X}\) consisting of clopen sets.
\end{proof}

The set of all finite subsets of $\N$ is denoted by $[\mathbb N]^{<\infty}$.

 \begin{corollary}\label{Polish}
For any function \(f\in (\omega + 1)^{\omega + 1}\) the topology \(\mathcal{T}_f\) is Polish and possesses a basis consisting of clopen sets.
\end{corollary}
\begin{proof}
 Since $\Tau_f$ is generated by $\bigcup_{n\in\N}\bigcup_{X\in[\N]^{<\infty}}\Tau_{n, X}$, Lemma 13.3 from~\cite{Kechris} implies that the space $(I_\N,\Tau_f)$ is Polish. By \cref{newlabel1}, for each $n\in\N$ and $X\in [\N]^{<\infty}$ the topology $\Tau_{n, X}$ has a basis $\mathcal B_{n,X}$ consisting of clopen sets. It is routine to verify that the family $\mathcal B=\bigcup_{n\in\N}\bigcup_{X\in[\N]^{<\infty}}\mathcal B_{n, X}$ is a subbasis of $\Tau_f$ consisting of clopen (with respect to $\Tau_f$) sets. Then the finite intersections of elements of $\mathcal B$ form a basis of  
 \(\mathcal{T}_f\) consisting of clopen sets.
 \end{proof}

\begin{theorem}\label{theorem-waning-function-to-Polish-topology}
For each waning function $f$, $(I_\N,\mathcal{T}_f)$ is a Polish topological semigroup.   
\end{theorem}

\begin{proof}
 By \cref{Polish}, the space $(I_\N,\mathcal{T}_f)$ is Polish.  Fix an arbitrary $a,b,c\in I_\N$ such that $ab=c$. If $(|c|)f=\w$, then by \cref{remarknew} the neighbourhoods of $c$ in $\Tau_f$ and $\Itwo$ coincide. Since $\Itwo$ is a semigroup topology and $\Itwo\subseteq \Tau_f$ we get that the semigroup operation is continuous at the point $(a,b)$. Assume that $(|c|)f\in\w$.  Fix an arbitrary open neighborhood $U$ of $c$. By \cref{neighbourhood_basis}, there exists $r\in\N$ such that $W_{f,c,r}\subseteq U$. Choose a positive integer $p$ that satisfies the following conditions:
 \begin{enumerate}
    \item $(p)f\leq (|a|)f=(|a{\restriction}_p|)f$;
    \item $(p)f\leq (|b|)f=(|b{\restriction}_p|)f$;
    \item $p\geq \max(\{0,\ldots,r\}a)$;
    \item $p\geq \max(\{0,\ldots,r\}b^{-1})$.
 \end{enumerate}
\cref{neighbourhood_basis} and conditions (1) and (2) imply that the sets $W_{f,a,p}$ and $W_{f,b,p}$ are open neighborhoods of $a$ and $b$, respectively. Let us show that $W_{f,a,p}\cdot W_{f,b,p}\subseteq W_{f,c,r}$.  Fix any elements $d\in W_{f,a,p}$ and $e\in W_{f,b,p}$. Since $a{\restriction}_p=d{\restriction}_p$ and $b{\restriction}_p=e{\restriction}_p$, condition (3) yields that $(de){\restriction}_r=c{\restriction}_r$.  Let 
$$A=\im(b)\cap\im(de)\cap (r\setminus\im(c))\quad\hbox{ and }\quad B=\im(e)\cap (r\setminus\im(b)).$$ We claim that $\im(de)\cap (r\setminus\im(c))\subseteq A\cup B$. Indeed, for any $x\in \im(de)\cap (r\setminus\im(c))$, either $x\in A$ or $x\in r\setminus \im(b)$. In the latter case, $x\in B$, as required. 
In order to show $de\in W_{f,c,r}$ it suffices to check $|A|+|B|\leq (|c|)f$.


Consider any $y\in A$.
 If $(y)b^{-1}\in \im (a)$, then $y\in \im(c)$, which contradicts the choice of $y$. By condition (4), $(y)b^{-1}\leq p$. Hence $(y)b^{-1}\in \dom(b)\cap (p\setminus \im(a))$, witnessing that $(A)b^{-1}\subseteq \dom(b)\cap (p\setminus \im(a))$. Since $b$ is a partial bijection, $|A|=|(A)b^{-1}|\leq |\dom(b)\cap (p\setminus \im(a))|$. For the sake of brevity we put $t=|\dom(b)\cap (p\setminus \im(a))|$. 
 Since $(y)b^{-1}\leq p$ and $e{\restriction}_p=b{\restriction}_p$ we get that $((y)b^{-1}, y)\in e$. Taking into account that $y\in\im(de)$, we obtain $(y)b^{-1}\in \im(d)\cap (p\setminus \im(a))$. It follows that $(A)b^{-1}\subseteq \im(d)\cap (p\setminus \im(a))$.
 Since $d\in W_{f,a,p}$ and $b$ is a partial bijection, we get that $|A|=|(A)b^{-1}|\leq (|a|)f$.
 Hence $$|A|\leq \min\{t, (|a|)f\}.$$

 Since  $e\in W_{f,b,p}$, $|B|\leq (|b|)f$. 

  Finally, let us check that $\min\{t,(|a|)f\} + (|b|)f\leq (|c|)f$. Since for every $x\in \dom(b)\cap (p\setminus \im(a))$, $(x)b\in \im(b)\setminus \im(c)$ we get that $|c|\leq |b|-t$. If $(|b|)f>0$, then $(|b|)f+t\leq (|c|)f$ as $f$ is a waning function. 
 If $(|b|)f=0$, then  
$\min\{t,(|a|)f\}+(|b|)f\leq (|a|)f\leq (|c|)f$, as $|a|\geq |c|$ and $f$ is a waning function. Thus, $de\in W_{f,c,r}$ and, consequently, $W_{f,a,p}\cdot W_{f,b,p}\subseteq W_{f,c,r}$. Hence $(I_\N,\Tau_f)$ is a Polish topological semigroup.   
\end{proof}

Finally, in terms of the promises made at the beginning of this section, we now show how the topologies $\mathcal{T}_f$ are ordered with respect to inclusion. 

\begin{lemma}\label{lem-order-preserving} 
If $f$ and $g$ are waning functions, then the following are equivalent:
\begin{enumerate}[\rm (i)]
    \item \(\mathcal{T}_f\subseteq  \mathcal{T}_g\).
    \item for all \(n\in \omega\), we have \((n)g\leq (n)f\).
\end{enumerate}
\end{lemma}
\begin{proof}
(i) $\Rightarrow$ (ii): 
Seeking a contradiction, we suppose that \(\mathcal{T}_f\subseteq  \mathcal{T}_g\) but that there exists $n\in \N$
such that \((n)f< (n)g\). 

Let $\id_n$ be the identity function on \(\{0, 1, \ldots, n-1\}\), let $b\in \N$ be such that $(n)f <b - n \leq (n)g$. 
By \cref{much_wan}, the set 
\(W_{f, \id_n, b}\) is open in \((I_\N, \Tau_f)\). 
Since \(W_{f, \id_n, b}\in \mathcal{T}_g\), by \cref{neighbourhood_basis} there is some \(r\in \N\) such that $W_{g, \id_n, r}\subseteq W_{f, \id_n, b}$.
Since $W_{g, \id_n, t + 1} \subseteq W_{g, \id_n, t}$ for any large enough $t$, we may choose $r > b$. By assumption $b- n > (n)f \geq 0$, $b > n$ and so
we may define:
\[h= \id_n \cup \makeset{(r+i, n+i)}{\(i\in \{0, 1,\ldots, b-n-1\}\)}\in I_\N.\]
Then $h{\restriction}_{r}=\id_n{\restriction}_r\) and \(|\im(h) \cap (b\setminus \im(\id_n)) | = |\im(h) \cap (r\setminus \im(\id_n))|=  b - n\). In particular,  \[(n)f < |\im(h) \cap (b\setminus \im(\id_n)) | = |\im(h) \cap (r\setminus \im(\id_n))| \leq (n)g = (|\id_n|)g\] and so $h \in W_{g, \id_n, r}$ but \(h\notin W_{f, \id_n, b}\),
which is a contradiction.

\noindent (ii) $\Rightarrow$ (i): Let \(U\in \mathcal{T}_f\). We show that \(U\) is an open neighbourhood of all its elements with respect to \(\mathcal{T}_g\).
Suppose that \(h\in U\). By \cref{neighbourhood_basis}, there is \(r\in \N\) such that \(h\in W_{f, h, r} \subseteq U\). Enlarging $r$ if necessary we can assume that $h\in W_{g, h, r}\in\Tau_g$. 
Since $(|h|)g \leq (|h|)f$, 
$$W_{g, h, r}
       \subseteq \{l\in I_\N: l{\restriction}_{r}=h{\restriction}_r\hbox{ and }|\im(l) \cap (r\setminus \im(h)) |\leq  (|h|)f\}=W_{f, h, r}.$$
Hence \(U\) is an open neighbourhood of \(h\) with respect to \(\Tau_g\).
\end{proof}

\cref{lem-order-preserving} implies the following.
\begin{corollary}\label{newcol}
If $f$ and $g$ are waning functions, then 
$\mathcal{T}_f=\mathcal{T}_g$ if and only if $f=g$.
\end{corollary}


The following corollary will be used in \cref{section-from-topologies-to-waning}.

\begin{corollary}\label{cor-nbhds-zero}
 Let  \(f, g:\omega+1^{\omega +1}\) be waning functions, 
 and $\mathcal{F}$ and $\mathcal{G}$ be
 the sets of neighbourhoods of \(\varnothing\) in $\mathcal{T}_f$ and $\mathcal{T}_g$, respectively.
 Then $\mathcal{F}\subseteq \mathcal{G}$ if and only if \((0)f\geq (0)g\).
\end{corollary}
\begin{proof}
By \cref{neighbourhood_basis}, the family $B_f(\varnothing) = \{W_{f, \varnothing, r}: (r)f \leq (0)f \}$ is a filter basis for $\mathcal{F}$. Since $f$ is a waning function, $(r)f\leq (0)f$ holds for all $r\in \omega + 1$, and so 
$B_f(\varnothing) = \{W_{f, \varnothing, r}: r\in \omega\}$.

If $(0)f \geq (0)g$, then $W_{f, \varnothing, r} \supseteq W_{g, \varnothing, r}$  for all $r$, by the definition of these sets (see \cref{definition}). So 
$B_f(\varnothing) =
\{W_{f, \varnothing, r}: r\in \omega  \}$
is coarser than 
$\{W_{g, \varnothing, r}: r \in \omega \}=
B_g(\varnothing)$. Hence $\mathcal{F}\subseteq \mathcal{G}$.

Conversely, suppose that $\mathcal{F}\subseteq \mathcal{G}$. If  we define waning functions $f'$ and $g'$ such that $(0)f' = (0)f$ and $(0)g' = (0)g$ and $(x)f' = (x)g' = 0$ if $x\neq 0$, 
then the set of neighbourhoods of $\varnothing$ in $\mathcal{T}_{f'}$ is $\mathcal{F}$. Similarly for $\mathcal{T}_{g'}$ and $\mathcal{G}$.
A subset $N$ of $I_{\N}$ is a neighbourhood in $\mathcal{T}_{f'}$ of $p\in I_{\N}\setminus \{\varnothing\}$ if and only if $N$ is a neighbourhood of $p$ in $\Ifour$.
Similarly, for $T_{g'}$.
Hence every neighbourhood in $\mathcal{T}_{f'}$ of every $p\in I_{\N}$ is also a neighbourhood of $p$ in $\mathcal{T}_{g'}$.
 So $\mathcal{T}_{f'} \subseteq \mathcal{T}_{g'}$ and so, by \cref{lem-order-preserving}, $(0)f = (0)f' \geq (0)g' = (0)g$.
\end{proof}

By \cref{Polish}, every function $f:\omega +1 \to \omega + 1$ (not just waning ones) corresponds to a Polish topology $\mathcal{T}_f$ on $I_{\N}$. It is natural (but not necessary) to ask how these compare to the Polish semigroup topologies $\mathcal{T}_f$ arising from waning functions $f$. 
We will show in \cref{from-function-to-waning-function} that $\mathcal{T}_f$ is a Polish semigroup topology for every function $f:\omega +1 \to \omega + 1$, but that there always exists a waning function $f'$ such that $\mathcal{T}_f=\mathcal{T}_{f'}$ and so non-waning functions do not give rise to any additional topologies. 

\begin{lemma}\label{lem:tech}
Let \(f\) be a waning function with \((\omega)f=0\),
 $n\in \N$ and  $X\in [\N]^{<\infty}$.  If \(g\in I_\N\) is arbitrary with \(|\im(g)\setminus X|\geq n\) and \(|X\cap \im(g)|\geq (n)f\), then \((|g|)f=0\).  
\end{lemma}

\begin{proof}
 If \(g\) is infinite, then we are done by the assumption that \((\omega)f=0\). Suppose that $g$ is finite. Then 
\[(|g|)f\leq \max(0, (|g|-1)f-1)\leq\cdots \leq \max(0, (|g|-(|g|-n))f-(|g|-n))=\max(0, (n)f-(|g|-n)).\]

As \(|\im(g)\setminus X| \geq n\), it follows that \(|g|-n\geq |\im(g)\cap X|\geq (n)f\). So \((n)f-(|g|-n)\leq 0\) and hence \(\max(0, (n)f-(|g|-n))= 0\) as required. 
\end{proof}

\begin{theorem}\label{from-function-to-waning-function}
      If \(f\in(\omega+1)^{(\omega+1)}\), then \(\Tau_f=\Tau_{f'}\). 
\end{theorem}
\begin{proof}
Note that if \(f{\restriction}_\omega\) is constant with value \(\omega\), then \(f{\restriction}_{\omega}=f'{\restriction}_{\omega}\) and we are done.
So assume that this is not the case which, in particular, implies that \((\omega)f'=0\).

    By \cref{non-wan_vs_wan_1}, we have \(\Tau_f\supseteq\Tau_{f'}\) and so it only remains to show that \(\Tau_{f'}\supseteq\Tau_{f}\). 
    Let \(n\in \N\), \(X\in[\N]^{<\infty}\), and 
    \[U_{f, n, X}= \makeset{g\in I_\N}{\(|\im(g) \setminus X|\geq n\) and \(|X\cap \im(g)|\leq (n)f\)}.\]
    In order to show that \(U_{f,n,X}\in \Tau_{f'}\), let \(g\in U_{f, n, X}\) be arbitrary. We show that \(U_{f, n, X}\) is a neighbourhood of \(g\) with respect to \(\mathcal{T}_{f'}\).


First assume that \((|g|)f'>0\).  Since \(|\im(g)\setminus X|\geq n\) and $f'$ is a waning function with $(\omega)f'=0$, it follows from the contra-positive of \cref{lem:tech}, that $|X \cap \im(g)| < (n)f'$ and so $g \in U_{f',n,X}$. Note that \(U_{f',n, X}\subseteq U_{f,n, X}\) so we have that \(U_{f,n, X}\) is a neighbourhood of \(g\) with respect to \(\Tau_{f'}\).

So now assume that \((|g|)f'=0\).
Let \(r\) be larger than all elements of \((X)g^{-1}\) as well as the first \(n\) elements of \((\N\setminus X)g^{-1}\), and large enough that \(W_{f',g, r}\in \Tau_{f'}\). It follows that
\[g\in W_{f',g, r}\subseteq \makeset{h\in I_{\N}}{\(g{\restriction}_r = h{\restriction}_r\) and \(\im(h)\cap X=\im(g)\cap X\)}\subseteq U_{f, n, X}\]
so \(U_{f, n, X}\) is a neighbourhood of \(g\) with respect to \(\Tau_{f'}\) as required. \end{proof}   

\section{From semigroup topologies to waning functions}  \label{section-from-topologies-to-waning}

We have shown in Section \ref{section-waning-func-to-topology} that every waning function gives rise to a Polish semigroup topology on $I_{\N}$. In this section we will prove the converse, namely that every Polish (in fact every $T_1$ and second-countable) semigroup topology on $I_{\N}$ corresponds to a waning function. In other words, the aim of this section is to finish the proof of \cref{main-theorem}.

Recall that there are two minimal Polish semigroup topologies on $I_{\N}$, namely $\Itwo$ and $\Ithree=\Itwo^{-1}$. Reflecting this fact, the correspondence between waning functions and Polish topologies on $I_{\N}$ is ``one-to-two" in the following sense. Every waning function $f$ gives rise to a unique Polish topology $\Tau_f$ on $I_{\N}$ which contains $\Itwo$ and a unique topology $\Tau_f^{-1}$ containing $\Ithree$. Due to this duality we will mostly consider an arbitrary $T_1$ and second-countable topology $\Tau$ containing $\Itwo$.

The aim of the next definitions and results (up to and including \cref{increasing_neighbourhoods}) is to construct a certain sequence of filters on $I_\N$ for such a $\Tau$ which we will call good filters. The good filters will later be used to define the waning function corresponding to $\Tau$.

\begin{definition}
If \(X\subseteq \N\) is finite, then we define \(i_X:\N\to \N \setminus X\) to be the unique order isomorphism.
For $g\in I_\N$ let ${\uparrow}g=\makeset{h\in I_\N}{\(g\subseteq h\)}$.
    If \(g\in I_\N\) is finite, then we define a bijection
    \[d_g:{\uparrow}g \to I_\N\quad \text{ by }\quad(h)d_g=i_{\dom(g)} \circ h \circ i_{\im(g)}^{-1}.\]

    Note that the function \(d_g\) is continuous with respect to any shift continuous topology on \(I_\N\) and has a clopen domain with respect to \(\Itwo\).
\end{definition}

\begin{lemma}\label{addendum}
 If $g$ is a finite idempotent of \(I_\N\), then $d_{g}$ is an isomorphism.   
\end{lemma}
\begin{proof}
Since $g$ is an idempotent, then ${\uparrow}g$ is an inverse subsemigroup of $I_\N$.
As \(d_g\) is a bijection it suffices to show that $d_g$ is a homomorphism.
    Let \(f, h\in {\uparrow}g\). Note that \(\im(i_{\dom(g)}f)\cap \im(g)=\varnothing\), as $g\subseteq f$ and $f$ is injective. It follows that $i_{\dom(g)}f i_{\im(g)}^{-1}i_{\im(g)}=i_{\dom(g)}f$. Then
    \[(f)d_g(h)d_g=i_{\dom(g)}f i_{\im(g)}^{-1}i_{\im(g)}h i_{\im(g)}^{-1}=i_{\dom(g)}f h i_{\im(g)}^{-1}=(fh)d_g.\]
\end{proof}

\begin{definition}
Let $\Tau$ be a second countable semigroup topology on \(I_{\N}\) containing \(\Itwo\) and \(g\in I_{\N}\) be finite. Then let
\[\mathcal{F}_{\Tau,g}:=\makeset{(N\cap {\uparrow}g)d_{g}}{\(N\) is a neighbourhood of \(g\) with respect to \(\Tau\)}.\]
\end{definition}


\begin{lemma}\label{filter_independant_of_size}
    If \(\Tau\) is a second countable semigroup topology on \(I_{\N}\) containing \(\Itwo\), and \(g, h\in I_{\N}\) are such that \(|g|=|h|<\w\), then
    \(\mathcal{F}_{\Tau,g}=\mathcal{F}_{\Tau,h}.\)
\end{lemma}
\begin{proof}
Let \(\theta_1\) and  \(\theta_2\) be the unique topologies on \(I_\N\) such that \(d_g:({\uparrow}g, \Tau) \to (I_\N, \theta_1)\) and \(d_h:({\uparrow}h, \Tau) \to (I_\N, \theta_2)\) are homeomorphisms.
Note that \(\mathcal{F}_{\Tau, g}\) is the set of neighbourhoods of \(\varnothing\) with respect to \(\theta_1\), and similarly \(\mathcal{F}_{\Tau, h}\) is the set of neighbourhoods of \(\varnothing\) with respect to \(\theta_2\).



    It suffices to show that \(\theta_1=\theta_2\).  
    Let \(k: \dom(g) \to \dom(h)\) be a bijection.
     Let \[k_d=k \cup (i_{\dom(g)}^{-1} \circ i_{\dom(h)})\in \Sym(\N).\] Set \[k_i=g^{-1}k_dh \cup (i_{\im(g)}^{-1} \circ i_{\im(h)}) \in \Sym(\N).\] Then \(\phi:I_\N \to I_\N\) defined by \((x)\phi= k_d^{-1}xk_i\) is a homeomorphism from \((I_\N, \Tau)\) to \((I_\N, \Tau)\). Observe that 
     $$(g)\phi=k_d^{-1}gk_i= k^{-1}gk_i= k^{-1}gg^{-1}k_dh=k^{-1}k_dh=k^{-1}kh=h.$$
    Also, for all \(f\in I_\N\) we have
\[(f)d_g^{-1}=i_{\dom(g)}^{-1}fi_{\im(g)}  \cup g.\] 
Hence
\begin{align*}(f)d_g^{-1}\phi=&(i_{\dom(g)}^{-1}fi_{\im(g)}\cup g)\phi\\
=&\left((i_{\dom(g)}^{-1} \circ i_{\dom(h)})^{-1}i_{\dom(g)}^{-1}fi_{\im(g)} (i_{\im(g)}^{-1} \circ i_{\im(h)})\right) \cup (g)\phi\\
=&\left( i_{\dom(h)}^{-1} f i_{\im(h)} \right) \cup h.
\end{align*}
Thus
\[(f)d_g^{-1}\phi d_h=i_{\dom(h)}\left(\left(i_{\dom(h)}^{-1}fi_{\im(h)} \right) \cup h\right)i_{\im(h)}^{-1}=f.\] 
So \(d_g^{-1}\phi d_h:(I_\N, \theta_1)\to (I_\N, \theta_2)\) is the identity function. Being a composition of homeomorphisms it is also a homeomorphism and so \(\theta_1=\theta_2\).
\end{proof}

\begin{definition}\label{good_filter_def}
    From now on we denote the filter \(\mathcal{F}_{\Tau,g}\) by \(\mathcal{F}_{\Tau,|g|}\)
    (this is well defined by \cref{filter_independant_of_size}). We say that a filter \(\mathcal{F}\) on $I_\N$ is {\em good} if  \(\mathcal{F}= \mathcal{F}_{\Tau, n}\) for some \(n\in \N\) and some second countable semigroup topology on \(I_\N\) containing \(\Itwo\).
    %
\end{definition}
\begin{lemma}\label{newlabel}
    If \(\mathcal{F}\) is a filter, then \(\mathcal{F}\) is good if and only if there is a second countable semigroup topology \(\mathcal{T}\) on \(I_\N\) containing \(\Itwo\) such that \(\mathcal{F}\) is the set of neighbourhoods of \(\varnothing\).
\end{lemma}
\begin{proof}
\((\Leftarrow):\) In this case \(\mathcal{F}=\mathcal{F}_{\mathcal{T}, \varnothing}\).

    \((\Rightarrow):\) Suppose that \(\mathcal{F}=\mathcal{F}_{\mathcal{T},n}\) for some \(n\in \N\). 
Let \(\id_n\) be the identity function on \(n\). 
By \cref{addendum}, \(d_{\id_n}\) is an isomorphism from \({\uparrow}\id_n\) to \(I_\N\) which maps \(\id_n\) to \(\varnothing\). 
Let the semigroup ${\uparrow}\id_n$ have the subspace topology inherited from $(I_\N,\Tau)$, and
\(\theta\) be the unique topology on \(I_\N\) such that 
\(d_{\id_n}: {\uparrow}\id_n \to (I_\N, \theta)\) is a topological isomorphism. Clearly, $\theta$ is a second countable semigroup topology. Since ${\uparrow}\id_n$ is open in $(I_\N, \Tau)$ we get that \((\mathcal{F}_{\mathcal{T},n})d_{\id_n}^{-1}\) is the set of all neighbourhoods of \(\id_n\) in ${\uparrow}\id_n$. Thus  \(\mathcal{F}_{\mathcal{T},n}=((\mathcal{F}_{\mathcal{T},n})d_{\id_n}^{-1})d_{\id_n}\) is the set of all neighbourhoods of $\varnothing$ in $(I_\N,\theta)$.
Hence \(\mathcal{F}_{\theta, 0}= \mathcal{F}_{\Tau, n}\).
\end{proof}

\begin{definition}\label{de-data}
    If \(\mathcal{T}\) is a second countable semigroup topology on \(I_{\N}\) containing \(\Itwo\), then we define the {\em filter data} of \(\mathcal{T}\) to be the sequence
    \[\mathfrak{F}_\mathcal{T}:=(\mathcal{F}_{\mathcal{T}, n})_{n\in \omega}.\]
\end{definition}

\begin{lemma}\label{increasing_neighbourhoods}
    If \(\mathfrak{F}_{\Tau}=(\mathcal{F}_{\Tau,n})_{n\in \omega}\) is the filter data of a second countable semigroup topology \(\Tau\supseteq \Itwo\) on \(I_{\mathbb{N}}\), then \(\mathcal{F}_{\Tau,n} \subseteq \mathcal{F}_{\Tau,m}\) for all \(n\leq m\).
\end{lemma}
\begin{proof}
    Assume without loss of generality that  \(n<m\).
Let \(k\in I_\N\) be such that \(\operatorname{id}_n \subseteq k\) and \(k\) bijectively maps \(\N\setminus 
 (m\setminus n)\) to \(\N\) in an order preserving fashion.
Observe that the function \(\phi_k:{\uparrow}\operatorname{id}_m \to{\uparrow}\operatorname{id}_n\) defined by
\[(x)\phi_k=k^{-1}xk\]
is a continuous bijection with inverse \(x\mapsto (kxk^{-1})\cup \operatorname{id}_{m\setminus n}\).  Since \(\operatorname{id}_{\N\setminus 
 (m\setminus n)}\operatorname{id}_m=\operatorname{id}_n\) and $k{\restriction}_n=\id_n$, then
\[(\operatorname{id}_m)\phi_k=k^{-1}\operatorname{id}_m k=\operatorname{id}_n \operatorname{id}_m \operatorname{id}_n=\operatorname{id}_n.\]

By \(\theta\) we denote the subspace topology on ${\uparrow} \operatorname{id}_n$ inherited from $(I_\N,\Tau)$. 
Hence if \(N\in \mathcal{F}_{\Tau,n}\), then \((N)d_{\operatorname{id}_n}^{-1}\) is a neighbourhood of \(\operatorname{id}_n\) with respect to  \(\theta\) and hence \((N)d_{\operatorname{id}_n}^{-1}\phi_k^{-1}\) is a neighbourhood of \(\operatorname{id}_m\) with respect to  \(\theta\).
To show \(N\in \mathcal{F}_{\Tau,m}\), it therefore suffices to check \((N)d_{\operatorname{id}_n}^{-1}\phi_k^{-1} d_{\operatorname{id}_m}=N\).

We show in fact that \(\phi_k^{-1} d_{\operatorname{id}_m}= d_{\operatorname{id}_n}\). Note that \(\phi_k, d_{\operatorname{id}_m}, d_{\operatorname{id}_n}\) are all injections and we have by definition that
\[\dom(\phi_k)= {\uparrow}\operatorname{id}_m,\quad \im(\phi_k)={\uparrow}\operatorname{id}_n, \quad \hbox{and } \quad
\im(d_{\operatorname{id}_m})= \im(d_{\operatorname{id}_n})=I_\N.\]
Thus \(\phi_k^{-1}\) is a bijection from \({\uparrow}\operatorname{id}_n\) to \({\uparrow}\operatorname{id}_m\).
For all \(x\in {\uparrow}\operatorname{id}_n\) we have
\[(x)d_{\operatorname{id}_n}=i_{n}x i_{ n}^{-1} \text{ and }(x)\phi^{-1}_k d_{\operatorname{id}_m}= i_{m}(kxk^{-1} \cup \operatorname{id}_{m\setminus n})i_{m}^{-1} =( i_{m}k)x( i_{m}k)^{-1} .\]
Thus we need only show that \(i_{n}=i_{m}k\). This is immediate as \(i_{n}:\N \to \N\setminus n\), \(i_{m}:\N \to \N\setminus m\), and \(k{\restriction}_{\N\setminus m}: \N\setminus m \to \N \setminus n \) are all order isomorphisms.
\end{proof}

The next major step is to show in \cref{wany} that every good filter has a convenient basis (we will call it a basis consisting of wany sets). Furthermore, we will show that if this basis is particularly nice (consisting of finitely wany sets), then the good filter arises from a waning function.

Recall that by \([X]^{<\infty}\) we denote the set of all finite subsets of a set \(X\).

\begin{definition}
    If \(\mathfrak{Y}\subseteq [\N]^{<\infty}\) is non-empty, then we say that a set \(S\subseteq I_\N\) is \(\mathfrak{Y}\)-{\em wany} if there is \(n\in \N\) such that
    \[S= N_{n,\mathfrak{Y}}:=\makeset{f\in I_{\N}}{\( \dom(f)\cap n=\varnothing\) and there is \(Y\in \mathfrak{Y}\) with \(\im(f)\cap Y=\varnothing\)}.\]
    We say a set is {\em wany} if it is \(\mathfrak{Y}\)-wany for some \(\mathfrak{Y}\).
    We say that a set is {\em finitely wany} if it is \(\mathfrak{Y}\)-wany for some finite \(\mathfrak{Y}\).
\end{definition}
\begin{remark}\label{remark}
    Note that  for example if \(\mathfrak{Y}=\{\varnothing\}\), then
    \[N_{n,\mathfrak{Y}}:=\makeset{f\in I_{\N}}{\( \dom(f)\cap n=\varnothing\)},\]
    if \(\mathfrak{Y}=\{\{0\}\}\), then
    \[N_{n,\mathfrak{Y}}:=\makeset{f\in I_{\N}}{\( \dom(f)\cap n=\varnothing\) \hbox{ and }\(0 \notin \im(f)\)},\]
     if \(\mathfrak{Y}=\{\{0,1 ,\ldots, k\}\}\), then
    \[N_{n,\mathfrak{Y}}:=\makeset{f\in I_{\N}}{\( \dom(f)\cap n=\varnothing\) \hbox{ and }\(\im(f)\cap \{0,1 \ldots, k\}=\varnothing\)},\]
   if \(\mathfrak{Y}=\makeset{Y\subseteq r}{\(|Y|= r-3\)}\), then
    \[N_{n,\mathfrak{Y}}:=\makeset{f\in I_{\N}}{\( \dom(f)\cap n=\varnothing\) and \(|\im(f)\cap r|\leq 3\)},\]
    etc.
    So all the usual sets we've used to define waning function topologies are of this form.    
\end{remark}

From here onward we denote the set of finite elements of \(I_\N\) by \(I_\N^{<\infty}\) (not to be confused with \([I_\N]^{<\infty}\) which is the set of finite subsets of \(I_\N\)).

\begin{definition}
    We say that a filter \(\mathcal{F}\) on \(I_\N\) is {\em defined by a waning function} \(f\), if \(\mathcal{F}\) is the filter of all neighbourhoods of \(\varnothing\) with respect to the topology \(\mathcal{T}_f\) (as defined in \cref{tau}).
\end{definition}

\begin{definition}
    If \(X\subseteq Y\) are sets and \(\mathcal{F}\) is a filter on \(Y\), then we define the {\em trace} of this filter on \(X\) by
    \[\mathcal{F}|_{X}:= \makeset{F\cap X}{\(F\in \mathcal{F}\)}.\]
\end{definition}

\begin{theorem}\label{wany}
If \(\mathcal{F}\) is a good filter on \(I_{\N}\), then the following hold:
\begin{enumerate}[\rm (i)]
    \item The filter \(\mathcal{F}\) has a filter base consisting of wany sets.
     \item If \(\mathcal{F}\) has a filter base consisting of finitely wany sets,
    then \(\mathcal{F}\) can be defined by a waning function.
         \item If the filter \(\mathcal{F}|_{I_{\N}^{<\infty}}\) on \(I_\N^{<\infty}\) has a  base of the form \((S_i\cap I_{\N}^{<\infty})_{i\in \N}\) where each set \(S_i\) is finitely wany, 
    then there is a good filter \(\mathcal{F}'\) defined by a waning function such that \(\mathcal{F}|_{I_{\N}^{<\infty}}=\mathcal{F}'|_{I_{\N}^{<\infty}}\).
\end{enumerate}
\end{theorem}
\begin{proof}
We start with point (1).
    Let \(\mathcal{F}\) be a good filter and \(N\in \mathcal{F}\) be arbitrary. 
    We will construct a wany set \(N''''\in \mathcal{F}\) with \(N''''\subseteq N\).
    By \cref{newlabel}, \(\mathcal{F}\) is the set of neighbourhoods of $\varnothing$ in a semigroup topology $\Tau$ on $I_\N$. Theorem~5.1.5(vii) from \cite{main} implies that  \(\Tau\subseteq \Ifour\). Then there exists an \(N'\in \mathcal{F}\cap \Ifour\) such that \(N' N' \subseteq N\).

    For all \((F, X, Y)\in  I_\N^{<\infty} \times [\N]^{<\infty} \times [\N]^{<\infty}\), define 
    \[U_{F, X, Y}:=\makeset{f\in I_{\N}}{\(F\subseteq f\), \(\dom(f)\cap X=\im(f)\cap Y=\varnothing\)}.\]
    As \(N'\in \Ifour\), there exists
    \(D\subseteq I_\N^{<\infty} \times [\N]^{<\infty} \times [\N]^{<\infty} \)
    such that
    \[N'=\bigcup_{(F, X, Y)\in D} U_{F, X, Y}.\]

Moreover, there must be some triple \((\varnothing, A, B) \in D\) as \(\varnothing\in N'\). Let \(n\in \N\) be such that \(A\cup B \subseteq n\) and define
\[D' =\makeset{(F, X\cup n, Y)}{\((F, X, Y)\in D\) and \(U_{F, X\cup n, Y} \neq \varnothing\)}.\]
Let
\[N'':=\bigcup_{(F, X, Y)\in D'} U_{F, X, Y}=\bigcup_{(F, X, Y)\in D} U_{\varnothing, n, \varnothing}\cap U_{F, X, Y}=N'\cap U_{\varnothing, n, \varnothing}.\]
Since $U_{\varnothing, n,\varnothing}$ is an open neighbourhood of $\varnothing$ in  $\Itwo\subseteq \Tau$, we get that $N''\in \mathcal{F}$. Let \(N''':=N''N''\subseteq N'N'\subseteq N\).
\begin{claim}\label{claim516}
    If \((F, X, Y)\in D'\), then 
    \[U_{\varnothing, n, B}U_{F, X, Y}=U_{\varnothing, n, Y}.\]
\end{claim}
\begin{proof}
    The containment \(\subseteq\) is immediate. Let \(h\in U_{\varnothing, n, Y}\) be arbitrary. 
    Note that \(U_{F, X, Y}\) is nonempty so \(\dom(F)\cap X= \im(F)\cap Y=\varnothing\).

    Let \(g_1=F\) and let \(g_2\) be any partial bijection with image equal to \(\im(h) \setminus \im(g_1)\) and domain disjoint from \(n\cup X\cup Y\cup \dom(F)\cup \im(F)\). Let \(g=g_1\cup g_2\in I_\N\).
    Note that \(g\in U_{F, X, Y}\) as \(\dom(F)\cap X=\im(F)\cap Y=\varnothing\) and \(\im(h)\cap Y=\varnothing\).
    
    Let \(f= hg^{-1}= h(g_1\cup g_2)^{-1}\).
    Note that \(\dom(f)\subseteq \dom(h)\), so 
    \[\dom(f)\cap n=\dom(h)\cap n=\varnothing \quad \text{and}\quad\im(f)\subseteq \dom(F)\cup \dom(g_2).\] 
    By the definition of \(D'\), \(X\supseteq n\supseteq B\), and so
    \begin{align*}\im(f)\cap B \subseteq (\dom(F)\cup \dom(g_2))\cap B=(\dom(F)\cap B) \cup (\dom(g_2)\cap B)=\\
    \dom(F)\cap B\subseteq \dom(F)\cap X= \varnothing.
    \end{align*}
    Hence \(f\in U_{\varnothing, n, B}\) and \(h=fg \in U_{\varnothing, n, B}U_{F, X, Y}\) as required.
\end{proof}

Let $$\mathfrak{Y}:= \makeset{Y\in [\N]^{<\infty}}{$Y$ \hbox{ appears as the third entry of an element of} $D'$}.$$ 
Since $(\varnothing, A, B) \in D$, we have \((\varnothing, n, B)\in D'\).
Thus by \cref{claim516}, we have 
\[\bigcup_{Y\in \mathfrak{Y}} U_{\varnothing, n, Y}=\bigcup_{(F, X, Y)\in D'} U_{\varnothing, n, B} U_{F, X, Y}\subseteq \left(\bigcup_{(F, X, Y)\in D'} U_{F, X, Y}\right)\left(\bigcup_{(F, X, Y)\in D'} U_{F, X, Y}\right)=N''N''=N''' \]
For convenience denote \(N'''':=\bigcup_{Y\in \mathfrak{Y}} U_{\varnothing, n, Y}\). Note that \(N''''\) is \(\mathfrak{Y}\)-wany. As $n\subseteq X$ for all \((F, X, Y)\in D'\) we have 
\( U_{F, X, Y}\subseteq U_{F, n, Y}\subseteq U_{\varnothing, n, Y}\). 
It follows that \(N''\subseteq N''''\) and subsequently \(N''''\in \mathcal{F}\). Since $N$ was chosen arbitrarily,  \(N''''\subseteq N'''\subseteq N\), and \(N''''\) is a wany set, the result follows.

We now prove the parts (2) and (3) of the theorem together by proving the following statement.
\medskip
\begin{enumerate}
     \item[(2)\&(3)]  If \(I\in \{I_\N, I_\N^{<\infty}\}\) and \(\mathcal{F}|_{I}\) has a filter base of the form \((S_i\cap I)_{i\in \N}\) where each set \(S_i\) is finitely wany, 
    then there is a good filter \(\mathcal{F}'\) defined by a waning function such that \(\mathcal{F}|_{I}=\mathcal{F}'|_{I}\).
\end{enumerate}
\medskip

Suppose that \(\mathcal{F}\) and \(I\) are as hypothesized. 
We define 
\[K=\makeset{l\in \N}{\(\makeset{f\in I}{\((l+1)\not\subseteq\im(f)\)}\in \mathcal{F}|_I\)}.\]
\begin{claim}\label{tidy_basis_claim}
    Let \(l, n\in \N\) and \(X\in [\N]^{<\infty}\) be arbitrary such that \(l<|X|\). Then 
    \[l\in K \iff \makeset{f\in I}{\(\dom(f)\cap n=\varnothing\) and \(|\im(f)\cap X|\leq l\)}\in \mathcal{F}|_I.\]
\end{claim}
\begin{proof}
    We fix \(l, n\in \N\) and \(X\in [\N]^{<\infty}\) such that \(l<|X|\). 
    
    \((\Leftarrow):\) Let \(\sigma\in \Sym(\N)\) be such that \(l+1\subseteq (X)\sigma\). Then
    \begin{align*}
       \big( \makeset{f\in I}{\(\dom(f)\cap n=\varnothing\) and \(|\im(f)\cap X|\leq l\)}\big)\sigma &\subseteq \makeset{f\in I}{ \(|\im(f)\cap (X)\sigma|\leq l\)}\\
        &\subseteq \makeset{f\in I}{\(\im(f)\not\supseteq (l+1)\)}.
    \end{align*}
Since $\F$ is a good filter and elements of $\Sym(\N)$ act bijectively on $I_\N^{<\infty}$ and they are homeomorphisms of $I_\N$ endowed with any shift-continuous topology, we get that the filter \(\mathcal{F}|_I\) is closed under right multiplication by elements of \(\Sym(\N)\). It follows that 
\[\makeset{f\in I}{\(\im(f)\not\supseteq (l+1)\)} \in \mathcal{F}|_I,\]
so \(l\in K\).

    \((\Rightarrow):\) We have that \(\makeset{f\in I}{\(|\im(f)\cap (l+1)|\leq l\)}\in \mathcal{F}|_I\).
    Let \(r=|X|\).
    Let $G_r$ be the finite subgroup of $\Sym(\N)$ consisting of all permutations with support contained in \(r\).
    Note that
    \begin{align*}
    \makeset{f\in I}{\(|\im(f)\cap r|\leq l\)}&=\makeset{f\in I}{\(\im(f)\cap r\) has no subset of size \(l+1\)}\\
    &=\makeset{f\in I}{there is no \(\sigma\in G_r\) such that \((\im(f))\sigma \supseteq l+1\)}\\
    &=\makeset{f\in I}{there is no \(\sigma\in G_r\) such that \(\im(f\sigma) \supseteq l+1\)}\\
          &= \bigcap_{\sigma\in G_r}\makeset{f\in I}{\(\im(f\sigma) \not\supseteq l+1\)}\\
           &= \bigcap_{\sigma\in G_r}\makeset{f\in I\sigma^{-1}}{\(\im(f\sigma) \not\supseteq l+1\)} \quad\quad(\text{as }I=I\sigma^{-1},\text{ because }\sigma\text{ is a bijection})\\
        &= \bigcap_{\sigma\in G_r}\makeset{g\sigma^{-1} \in I\sigma^{-1}}{\(\im(g\sigma^{-1}\sigma) \not\supseteq l+1\)}\\
                &= \bigcap_{\sigma\in G_r}\makeset{g\sigma^{-1} \in I\sigma^{-1}}{\(\im(g) \not\supseteq l+1\)}\\
  &= \bigcap_{\sigma\in G_r}\makeset{f\sigma^{-1} \in I}{\(\im(f) \not\supseteq l+1\)}\\
&= \bigcap_{\sigma\in G_r}\makeset{f\in I}{\(\im(f) \not\supseteq l+1\)}\sigma^{-1}.
    \end{align*}
So, since the filter \(\mathcal{F}|_I\) is closed under right multiplication by elements of \(\Sym(\N)\) we have that  \(\makeset{f\in I}{\(|\im(f)\cap r|\leq l\)}\in \mathcal{F}|_I\). As \(\mathcal{F}\) contains all the \(\Itwo\) neighbourhoods of \(\varnothing\), we obtain that
\begin{align*}
    \{f\in I|\dom(f)\cap n=\varnothing &\hbox{ and } |\im(f)\cap r|\leq l\}\\
    &=\makeset{f\in I }{ \(\dom(f)\cap n=\varnothing\)}\cap\makeset{f\in I}{\(|\im(f)\cap r|\leq l\)}  \in \mathcal{F}|_I.
\end{align*}
Let \(\rho\in \Sym(\N)\) be such that \((r)\rho=X\). It follows that
\begin{align*}
\{f\in I| \dom(f)\cap n=\varnothing &\hbox{ and }|\im(f)\cap X|\leq l\}\\
&=\makeset{f\in I}{\(\dom(f)\cap n=\varnothing\) and \(|\im(f)\cap r|\leq l\)}\rho\in \mathcal{F}|_I
\end{align*}
as required.
\end{proof}
\begin{claim}\label{basis_claim}
 If \(F\in \mathcal{F}|_I\), then there exist \(l, n\in \N\) and \(X\in [\N]^{<\infty}\) such that 
  \[U:=\makeset{f\in I}{\(\dom(f)\cap n=\varnothing\) and \(|\im(f)\cap X|\leq l\)}\in \mathcal{F}|_I\]
and $U \subseteq F$.
\end{claim}
\begin{proof}
Let \(F\in \mathcal{F}\) be arbitrary. By the assumption, \(\mathcal{F}|_{I}\) has a filter base of the form \((S_i\cap I)_{i\in \N}\) where each set \(S_i\) is finitely wany.  Hence there is a finite non-empty set \(\mathfrak{Y}\) and \(n\in \N\) such that
\(N_{n,\mathfrak{Y}}\cap I\in \mathcal{F}|_I\) and
\[N_{n,\mathfrak{Y}}\cap I=\makeset{f\in I}{\(\dom(f)\cap n=\varnothing\) and there is \(Y\in \mathfrak{Y}\) with \(\im(f)\cap Y=\varnothing\)}\subseteq F.\] 
Let \(l\) be the largest integer such that \(l\leq |\bigcup \mathfrak{Y}|\) and every subset of \(\bigcup \mathfrak{Y}\) of size \(l\) is disjoint from at least one element of \(\mathfrak{Y}\). Since $0$ satisfies the condition above, the integer $l$ is well defined.

Let \(G\) be the finite subgroup of \(\Sym(\N)\) consisting of all permutations supported on \(\bigcup \mathfrak{Y}\). Then define
\[M:= I\cap \bigcap_{\sigma\in G} N_{n,\mathfrak{Y}}\sigma.\]
Note that \(M\in \mathcal{F}|_I\), and if for some \(f\in I_\N\), \(|\im(f)\cap \bigcup\mathfrak{Y}|\leq l\), then \(\im(f\sigma)\) must be disjoint from an element of \(\mathfrak{Y}\) for all \(\sigma\in G\), so \(f\in M\).
Conversely, if \(f\in M\), then \(|\im(f)\cap \bigcup\mathfrak{Y}|\leq l\) as otherwise there is \(\sigma\in G\) such that \(\im(f\sigma)\) intersects every element of \(\mathfrak{Y}\) and hence \(f\not\in N_{n,\mathfrak{Y}}\sigma^{-1}\).
So 
\[M=\makeset{f\in I}{\(\dom(f)\cap n= \varnothing\) and \(|\im(f)\cap \bigcup \mathfrak{Y}|\leq l\) }\in \mathcal{F}|_I\]
choosing \(X=\bigcup \mathfrak{Y}\) the claim follows.
\end{proof}
If for all \(F\in \mathcal{F}|_I\), there is \(l, n\in \N\) and \(X\in [\N]^{<\infty}\) with \(l= |X|\) such that
  \[\makeset{f\in I}{\(\dom(f)\cap n=\varnothing\) and \(|\im(f)\cap X|\leq l\)}\in \mathcal{F}|_I\]
  and
  \[\makeset{f\in I}{\(\dom(f)\cap n=\varnothing\) and \(|\im(f)\cap X|\leq l\)}\subseteq F,\]
  then we have that \(\mathcal{F}|_I\) is the set of all neighbourhoods of \(\varnothing\) in the subspace topology on \(I\) inherited from \((I_\N, \Itwo)\), and we are done. Otherwise it follows from \cref{basis_claim} and \cref{tidy_basis_claim} that \(K\neq \varnothing\).
  Let \(k= \min(K)\). It then follows from the same two claims that the sets of the form
  \[\makeset{f\in I}{\(\dom(f)\cap n=\varnothing\) and \(|\im(f)\cap X|\leq k\)}\]
  for \(n\in \N\) and \(X\in [\N]^{<\infty}\) with \(|X|>k\) all belong to \(\mathcal{F}|_I\) and are a basis for this filter.
  Hence if \(\mathcal{F}'\) is any filter defined by a waning function \(\phi\) such that 
   \((0)\phi=k\), then \(\mathcal{F}|_I=\mathcal{F}'|_I\).
\end{proof}

We have shown that every good filter has a basis of wany sets and that a basis of finitely wany sets may be used to construct a waning function. To bridge the remaining gap, we show in \cref{finite_done} that a good filter cannot have a basis of wany sets without having a basis of finitely wany sets.

\begin{lemma}\label{inferror_implies_finite_error}
    If \(\mathcal{F}_1, \mathcal{F}_2\) are good filters  and \(\mathcal{F}_1|_{I_{\N}^{<\infty}}= \mathcal{F}_2|_{I_{\N}^{<\infty}}\), then \(\mathcal{F}_1= \mathcal{F}_2\).
\end{lemma}
\begin{proof}
Let \(\theta_1, \theta_2\) be second countable semigroup topologies on \(I_\N\) containing \(\Itwo\) such that \(\mathcal{F}_1\) is the set of \(\theta_1\) neighbourhoods of \(\varnothing\) and \(\mathcal{F}_2\) is the set of \(\theta_2\) neighbourhoods of \(\varnothing\).

Suppose for a contradiction that \(\mathcal{F}_1\not\subseteq\mathcal{F}_2\).
By \cref{wany}, we can fix a wany set \(N_{m,\mathfrak{Y}}\in \mathcal{F}_1\setminus \mathcal{F}_2\).
    Consider the product \(\varnothing\operatorname{id}_{\N}=\varnothing\) in the topological semigroup \((I_\N, \theta_1)\).
    Since \((I_\N, \theta_1)\) is a topological semigroup, we can find \(W_{l}\in \mathcal{F}_1\cap \theta_1\) and \(W_{r}\in \theta_1\subseteq \Ifour\) such that \(\id_\N \in W_r\) and \(W_l W_r\subseteq N_{m, \mathfrak{Y}}\). 
 Since open neighborhood bases at $\id_\N$ in $(I_\N,\Itwo)$ and $(I_\N,\Ifour)$ coincide, we lose no generality assuming there is \(n\geq m\) such that
 \(W_r=\uparrow \id_n\).

    As \(W_l\cap I_{\N}^{<\infty}\in \mathcal{F}_1|_{I_{\N}^{<\infty}}=\mathcal{F}_2|_{I_{\N}^{<\infty}}\),  we can (by \cref{wany}(1)) find a wany set \(N_{k, \mathfrak{Y}'}\in \mathcal{F}_2\) such that \(N_{k, \mathfrak{Y}'}\cap I_{\N}^{<\infty}\subseteq W_l\) and \(k\geq n\). 
    Since \(N_{m, \mathfrak{Y}}\notin\mathcal{F}_2\) and 
    $N_{k, \mathfrak{Y}'}\in \mathcal{F}_2$, we get that
    \(N_{k, \mathfrak{Y}'}\not\subseteq N_{m, \mathfrak{Y}}\). On the other hand, \((N_{k, \mathfrak{Y}'}\cap I_{\N}^{<\infty})\cdot {\uparrow} \id_n\subseteq W_lW_r\subseteq N_{m, \mathfrak{Y}}\).
    
    Let \(g\in N_{k, \mathfrak{Y}'}\setminus N_{m, \mathfrak{Y}}\). If $g$ is finite, then $g\in 
    (N_{k, \mathfrak{Y}'}\cap I_{\N}^{<\infty})\cdot \id_{\N}\subseteq (N_{k, \mathfrak{Y}'}\cap I_{\N}^{<\infty})\cdot {\uparrow} \id_n \subseteq  N_{m, \mathfrak{Y}}$, which is a contradiction. Hence
    \(g\) is infinite. Then \(g\not\in N_{m, \mathfrak{Y}}\) and \(k\geq m\geq n\) implies that \(\im(g)\) intersects every element of \(\mathfrak{Y}\). Hence \(N_{m, \mathfrak{Y}}\) contains no elements with the same image as \(g\). 
    
    Let \(g':=g{\restriction}_{(\{0,1 , \ldots, n-1\})g^{-1}}\).
    Since $g\in N_{k, \mathfrak{Y}'}={\downarrow}N_{k, \mathfrak{Y}'}:=\bigcup\{{\downarrow}x: x\in N_{k,\mathfrak{Y}'}\}$, it follows that 
    \(g'\in N_{k, \mathfrak{Y}'}\cap I_{\N}^{<\infty}\subseteq W_l\in \Ifour\). So \(W_l\cap {\uparrow} g'\in \Ifour\) and  \((W_l\cap {\uparrow} g')\cdot {\uparrow}\id_n\subseteq N_{m, \mathfrak{Y}}\).
Since \(g'\in W_l\cap {\uparrow} g'\in \Ifour\), there must be some infinite \(g''\in W_l\cap {\uparrow} g'\) with \(\im(g'')\cap n = \im(g') = \im(g)\cap n\).
So the set  \(g''\cdot {\uparrow} \id_n\) contains an element with the same image as \(g\).
This is a contradiction, as \(g''\cdot {\uparrow} \id_n\subseteq W_l'\cdot {\uparrow} \id_n\subseteq N_{m, \mathfrak{Y}}\), and \(N_{m, \mathfrak{Y}}\) contains no elements with the same image as \(g\). Therefore $\mathcal{F}_1\subseteq \mathcal{F}_2$ and by symmetry $\mathcal{F}_1 = \mathcal{F}_2$.
\end{proof}

\begin{definition}
    We say that a non-empty set \(\mathfrak{Y}\subseteq [\N]^{<\infty}\) is \textit{bad} if there is a good filter \(\mathcal{F}\) and \(n\in \N\) such that \(N_{n,\mathfrak{Y}}\in \mathcal{F}\) and there is no finitely wany set \(S\in \mathcal{F}\)  with \(S \subseteq N_{n,\mathfrak{Y}}\).
\end{definition}

Note that if $\mathfrak{Y}$ is bad and the good filter $\mathcal{F}$ witnesses this, then $\mathcal{F}$ has no filter base consisting of finitely wany sets. We will show (eventually) that there are, in fact, no bad sets. 

Recall that $[\N]^n$ denotes the set of all subsets of $\N$ of cardinality $n$.

\begin{lemma}\label{no bad}
    If there exists a bad set \(\mathfrak{Y}\subseteq [\N]^{<\infty}\), then there is a set \(\mathfrak{Y}'\subseteq [\N]^{<\infty}\) such that the following hold:
    \begin{enumerate}[\rm (1)]
        \item There is a good filter \(\mathcal{F}\) such that \(N_{0, \mathfrak{Y}'}\in \mathcal{F}\), and there is no finitely wany set \(S\), with \(S\cap I_{\N}^{<\infty}\in \mathcal{F}|_{I_\N^{<\infty}} \) and
        \(S\cap I_{\N}^{<\infty}\subseteq N_{0, \mathfrak{Y}'}\).
           \item The elements of \(\mathfrak{Y}'\) are incomparable with respect to containment and if \(Y\in [\N]^{<\infty}\) and \(\makeset{f\in I_\N^{<\infty}}{\(\im(f)\cap Y=\varnothing\)} \subseteq N_{0, \mathfrak{Y}'}\), then there is \(Y'\in \mathfrak{Y}'\) with \(Y'\subseteq Y\).
        \item  For all \(n\in \N\),  the set \([\N]^n\cap \mathfrak{Y}'\) is finite.
    \end{enumerate}
\end{lemma}
\begin{proof}
 We first construct a set satisfying (1), we will then modify this set to satisfy (2) while ensuring it still satisfies (1). The resulting set will satisfy (3) as well.
 Let \(\mathfrak{Y}\) be a bad set and let \(\mathcal{F}\) be a good filter witnessing that \(\mathfrak{Y}\) is bad.
 
 

By \cref{wany}, we can find a sequence \((N_{i, \mathfrak{Y}_i})_{i\in \N}\) of wany sets forming a filter base for \(\mathcal{F}\). 

For a good filter \(\mathcal{F}'\) and a set \(\mathfrak{Y}'\subseteq [\N]^{<\infty}\), let \(P(\mathcal{F}',\mathfrak{Y}')\) be the statement:

\begin{itemize}
    \item \(N_{0, \mathfrak{Y}'}\in \mathcal{F}'\), and there is no finitely wany set \(S\), with \(S\cap I_{\N}^{<\infty}\in \mathcal{F}'|_{I_\N^{<\infty}} \) and
        \(S\cap I_{\N}^{<\infty}\subseteq N_{0, \mathfrak{Y}'}\).
\end{itemize}

Seeking a contradiction, assume that for every \(i\in \N\) the statement \(P(\mathcal{F}, \mathfrak{Y}_i)\) is false. Then for all \(i\in \N\), we can find a finitely wany set \(S_i\) such that \(S_i \cap I_{\N}^{<\infty}\in \mathcal{F}|_{I_\N^{<\infty}}\) and \(S_i \cap I_{\N}^{<\infty}\subseteq N_{0, \mathfrak{Y}_i}\). 
Since $\mathcal{F}$ contains every $\Itwo$ neighbourhood of $\varnothing$, for all \(i\in \N\) 
the set \(U_i:=\makeset{f\in I_\N}{\(\dom(f)\cap i= \varnothing\)}\) belongs to \(\mathcal{F}\).
Hence replacing \(S_i\) by the finitely wany set \(S_i \cap U_i\) if needed, we can assume that \(S_i \cap I_{\N}^{<\infty}\subseteq N_{i, \mathfrak{Y}_i}\) for all \(i\in \N\).
Since $\makeset{N_{i, \mathfrak{Y}_i}}{\(i\in \N\)}$ is a filter base for $\F$, it follows that \(\makeset{S_i\cap I_{\N}^{<\infty}}{\(i\in \N\)}\) is a filter base for \(\mathcal{F}|_{I_\N^{<\infty}}\). By \cref{wany}(3), there is a filter \(\mathcal{F}'\) defined by a waning function such that \(\mathcal{F}'|_{I_\N^{<\infty}}=\mathcal{F}|_{I_\N^{<\infty}}\).
By \cref{inferror_implies_finite_error}, \(\mathcal{F}'=\mathcal{F}\). This is a contradiction as by assumption \(\mathcal{F}\) witnesses that \(\mathfrak{Y}\) is bad, and  \(\mathcal{F}\) has a filter base consisting of finitely wany sets (see \cref{remark}).

Thus we can find a set \(\mathfrak{A}\) such that \(P(\mathcal{F}, \mathfrak{A})\).
We next define a set $\mathfrak{Y}'$ using $\mathfrak{A}$ to satisfy conditions (1) and (2). First, however, let 
\[\mathfrak{A}':=\mathfrak{A} \cup \makeset{Y\in [\N]^{<\infty}}{ \(\makeset{f\in I_\N^{<\infty}}{\(\im(f)\cap Y=\varnothing\)} \subseteq N_{0, \mathfrak{A}}\)}.\]
Note that \(N_{0, \mathfrak{A}}\subseteq N_{0, \mathfrak{A}'}\) and \(N_{0, \mathfrak{A}}\cap I_{\N}^{<\infty}=N_{0, \mathfrak{A}'}\cap I_{\N}^{<\infty}\) and hence \(\mathfrak{A}'\) also satisfies (1). 
Let \(\mathfrak{Y}'\) be the set obtained from \(\mathfrak{A}'\) by removing all elements from \(\mathfrak{A}'\) which are not minimal with respect to containment.
It is easy to see that \(N_{0, \mathfrak{A}'}=N_{0, \mathfrak{Y}'}\). Note that \(\mathfrak{Y}'\) now satisfies (1) and (2). 

In order to show that $\mathfrak{Y}'$  satisfies (3), we need the following auxiliary fact. 
    \begin{claim}\label{claim-y-triple-prime}
   There are pairwise disjoint \(Y_1, \ldots, Y_k \in \mathfrak{Y}'\) such that \(Y\cap \left(\bigcup_{1\leq i\leq k}Y_i\right)\neq \varnothing\) for all \(Y\in \mathfrak{Y}'\).
\end{claim}
\begin{proof}
Note that if \(\varnothing\in \mathfrak{Y}'\), then  \(N_{0,\mathfrak{Y}'} =I_\N=N_{0, \{\varnothing\}}\).
Then the finitely wany set $N_{0,\mathfrak{Y}'}$ belongs to $\mathcal{F}$, and so setting $S =N_{0,\mathfrak{Y}'}$, we obtain that $\mathfrak{Y}'$ doesn't satisfy condition (1). The obtained contradiction implies  \(\varnothing\notin \mathfrak{Y}'\).

Suppose for a contradiction that the claim is false.
 Let \(Y_1\in \mathfrak{Y}'\) be arbitrary.
    If \(Y_1, \ldots, Y_l\in \mathfrak{Y}'\) are defined and pairwise disjoint, then as the claim is false we can find an \(Y_{l+1}\in \mathfrak{Y}'\) disjoint from all of them.
    So we have an infinite sequence of pairwise disjoint sets \(Y_1, Y_2, \ldots \in \mathfrak{Y}'\).
    If \(f\in I_{\N}^{<\infty}\), then \(\im(f)\) intersects at most finitely many of the \(Y_i\), hence \(N_{0,\mathfrak{Y}'}\cap I_{\N}^{<\infty}=I_{\N}^{<\infty}=N_{0,\{\varnothing\}}\cap I_{\N}^{<\infty}\). This contradicts the fact that \(\mathfrak{Y}'\) satisfies (1).
\end{proof}

We show inductively that:

\begin{itemize}
    \item[$\star$] For all \(n\in \N\), there is a finite set \(B_n\subseteq \N\) such that for all \(Y\in \mathfrak{Y}'\), we have \(|Y\cap B_n|\geq \min\{|Y|, n\}\).
\end{itemize}
 Note that this implies that if \(|Y|\leq n\) then \(Y\subseteq B_n\), in particular this will imply \((3)\) since \([\N]^n\cap \mathfrak{Y}'\subseteq \mathcal{P}(B_n)\).
 For the base case of the induction we set \(B_1= \bigcup_{1\leq i\leq k}Y_i\) where the sets $Y_1, \ldots, Y_k$ are those from \cref{claim-y-triple-prime}.

    Suppose inductively that the statement holds for some \(i\in \N\), so the set \(B_i\) is defined. 
    For every subset \(S\) of \(B_i\), let \[\mathfrak{Y}_S':= \makeset{Y\setminus B_i}{\(Y\in \mathfrak{Y}'\) and \(Y\cap B_i=  S\)}=\makeset{Y\setminus S}{\(Y\in \mathfrak{Y}'\) and \(Y\cap B_i=  S\)}.\] 

    Let \(S\subseteq B_i\) be arbitrary. If \(\mathfrak{Y}_S'\neq \varnothing\), let \(Y_{1, S}\in \mathfrak{Y}_S'\) (otherwise define \(Y_{1, S}:= \varnothing\)). If \(Y_{1, S}, \ldots, Y_{j, S}\) are defined, then let \(Y_{j+1, S}\in \mathfrak{Y}_S'\) be disjoint from \(Y_{1, S} \cup\ldots \cup Y_{j, S}\) (if this is impossible define \(Y_{j+1, S}:= \varnothing\)). 

\begin{claim}
    \(\bigcup_{j=1}^{\infty} Y_{j, S}\) is finite.
\end{claim}
\begin{proof}
Suppose for a contradiction that the set is infinite. It follows that \( Y_{j, S} \neq \varnothing\) for infinitely many \(j\). Without loss of generality suppose that every $Y_{j, S}$ is nonempty.
Note that, for all $j$,  $Y_{j, S} \cup S \in \mathfrak{Y}'$, 
and \(Y_{j, S}\cap B_i = \varnothing\).

 We show that 
\(\makeset{f\in I_\N^{<\infty}}{\(\im(f)\cap S=\varnothing\)} \subseteq N_{0, \mathfrak{Y}'}\) (so we can apply part (2) to $S$).
Let \(f\in I_{\N}^{<\infty}\) be such that \(\im(f) \cap S= \varnothing\).
Since \(\im(f)\) is finite and the sets \(Y_{j, S}\) are pairwise disjoint, only finitely many sets \(Y_{j, S}\) have nonempty intersection with \(\im(f)\). In other words, there exists \(j\in \N\) such that \(Y_{j, S} \cap \im(f)=\varnothing\).
It follows that \((Y_{j, S}\cup S) \cap \im(f)=\varnothing\), so \(f\in N_{0, \mathfrak{Y}'}\).
By (2) (using \(S\) as \(Y\)), there is a subset \(Y'\) of \(S\) belonging to \(\mathfrak{Y}'\).
    
Let \(Y\in \mathfrak{Y}_S'\) be arbitrary. Then \(Y'\subseteq Y\cup S\in \mathfrak{Y}'\)
    and so $Y', Y\cup S\in \mathfrak{Y}'$. But elements of \(\mathfrak{Y}'\) are incomparable  and so \(Y' \subseteq S \subseteq Y\cup S = Y'\).
    Since $Y$ and $S$ are disjoint, 
    $Y=\varnothing$ which shows that \(\mathfrak{Y}_S'\subseteq \{\varnothing\}\). This contradicts the assumption that the sets $Y_{j, S}$ are non-empty, as required.
\end{proof}

If \(Y\in \mathfrak{Y'}\) and \(Y\cap B_i= S\), then we have by construction that either \(Y\subseteq B_i\) or 
\[Y\cap \bigcup_{j=1}^{\infty} Y_{j, S} \neq \varnothing.\]
Thus the set 
\[B_{i+1}=B_i \cup \bigcup_{S\subseteq B_i}\bigcup_{j=1}^{\infty} Y_{j, S}\]
 satisfies $\star$.
   \end{proof}
   \begin{lemma}\label{bad_don't_exist}
       Bad sets do not exist. 
   \end{lemma}
   \begin{proof}
   Suppose for a contradiction that bad sets do exist, and let $\mathfrak{Y}$ be such a set. Let $\mathfrak{Y}'$ be the set from  \cref{no bad}, and let \(\mathcal{F}\) be the good filter from \cref{no bad}(1).
   By \cref{wany}(1), we may suppose that
    \[N_{1, \mathfrak{Y}_1}\supseteq N_{2,\mathfrak{Y}_2} \supseteq \cdots\]
    is a filter base for \(\mathcal{F}\) for some sets $\mathfrak{Y}_i\subseteq [\N]^{<\infty}$.
     Since \(\mathcal{F}\) is the set of neighbourhoods of \(\varnothing\) in a semigroup topology on \(I_\N\) and composition with an element of $\Sym(\N)$ is a homeomorphism, it follows that 
    \(\mathcal{F} \phi = \mathcal{F}\) for all \(\phi\in \Sym(\N)\),
    where 
    \[
    \mathcal{F}\phi = \{Af : A\in \mathcal{F}\}.\]  
    In particular, since $N_{0, \mathfrak{Y}'}\in \mathcal{F}$ it follows that $N_{0, \mathfrak{Y}'}\phi\in \mathcal{F}$ for any \(\phi\in \Sym(\N)\).
    Since the sets $N_{i, \mathfrak{Y}_i}$ are a base for $\mathcal{F}$, there exists 
     \(i_{\phi}\in \N\) such that \(N_{i_{\phi},\mathfrak{Y}_{i_\phi}}\subseteq N_{0,\mathfrak{Y}'} \phi\).   For all \(j\in \N\), let \(S_j:= \makeset{\phi\in \Sym(\N)}{\(i_{\phi}= j\)}\).
    Note that 
    \[\Sym(\N)= \bigcup_{i\in \N} S_i.\]
    As \(\Sym(\N)\) is Baire (under the pointwise topology), 
    at least one of the sets $S_j$ is not nowhere dense. In other words, 
    the closure of $S_j$ contains a neighbourhood of some $\theta \in \Sym(\N)$.

    
    If \(\phi \in \theta^{-1}S_j\), then $\theta\phi\in S_j$ and so from the definition of $S_j$, 
    \begin{equation}\label{eq-name}\tag{$\star$}
    \begin{array}{rcl}
    N_{j, \mathfrak{Y}_j} \subseteq N_{0, \mathfrak{Y}'}\theta\phi
    & =&  \{f\in I_\N: \exists Y \in \mathfrak{Y}', \im(f)\cap Y= \varnothing  \}\theta\phi \\
    & = & \{f\in I_\N: \exists Y \in \mathfrak{Y}', \im(f)\cap Y\theta= \varnothing  \}\phi \\
    & = &
    N_{0, \mathfrak{Y}'\theta}\phi.
    \end{array}
    \end{equation}
    Also \(\mathfrak{Y}'\theta\) satisfies the conditions of \cref{no bad} with respect to the same \(\mathcal{F}\) (as \(\mathcal{F}\theta= \mathcal{F}\) and \(\mathcal{F}|_{I_\N^{<\infty}}\theta= \mathcal{F}|_{I_\N^{<\infty}}\)).
    Since the closure of $S_j$  contains a neighbourhood of $\theta$, it follows that the closure of 
     \(\theta^{-1}S_j\) contains a neighbourhood of \(\operatorname{id}_\N\). In other words, the closure of \(\theta^{-1}S_j\) contains the pointwise stabilizer \(\makeset{f\in \Sym(\N)}{\(x\in F \Rightarrow (x)f=x\)}\) in \(\Sym(\N)\) of some finite set \(F\).
    \begin{claim}\label{sub-sub-sub}
        There exists \(Y\in \mathfrak{Y}_j\) such that there is no \(Y'\in \mathfrak{Y}'\theta\) with \(Y'\subseteq F\cap Y\).
    \end{claim}
    \begin{proof}
If the claim is false, then for all \(Y\in\mathfrak{Y}_j\) there is \(Y'\in \mathfrak{Y}'\theta\) such that \(Y'\subseteq F\cap Y\). 
Thus 
\[N_{j,\mathfrak{Y}_j}\cap I_{\N}^{<\infty}\subseteq \{f\in I_{\N}^{<\infty}: \im(f)\cap Y'=\varnothing \text{ for some } Y'\in \mathfrak{Y}'\theta\cap \mathcal{P}(F)\}\subseteq
N_{0, \mathfrak{Y}'\theta\cap \mathcal{P}(F)}\subseteq N_{0,\mathfrak{Y}'\theta}.\]
Then the set $S= N_{0, \mathfrak{Y}'\theta\cap \mathcal{P}(F)}$ is finitely wany, $S\cap I_{\N} ^ {<\infty}\in \mathcal{F}|_{I_{\N}^{<\infty}}$, and $S \cap I_{\N} ^ {<\infty}\subseteq N_{0, \mathfrak{Y}'\theta}$. But \cref{no bad}(1) applied to \(\mathfrak{Y}'\theta\) states that no such set $S$ exists, a contradiction.
    \end{proof}
    Let the set \(Y\in \mathfrak{Y}_j\) be as given in \cref{sub-sub-sub}.
    Let \(B\) be the union of all elements of \(\mathfrak{Y}'\theta\) of size at most \(|Y|\). By \cref{no bad}(3), the set $B$ is finite.
    Since $Y\in \mathfrak{Y}_j$, $Y$ is also finite. Hence there exists  
    \(\phi \in \Sym(\N)\) fixing \(F\) pointwise and satisfying \((B\setminus F)\phi \cap Y = \varnothing\).
    Since the closure of \(\theta^{-1}S_j\) contains the pointwise stabilizer of \(F\), we can choose  \(\phi\in\theta^{-1}S_j\).
    From \eqref{eq-name}, \(N_{0, \mathfrak{Y}_j}\subseteq N_{0, \mathfrak{Y}'\theta}\phi\), and \cref{no bad}(2) implies that there is \(Y'\in \mathfrak{Y}'\theta\phi\) such that \(Y'\subseteq Y\). In particular \(Y'\phi^{-1}\in \mathfrak{Y}'\theta\) and has at most \(|Y|\) elements. 
    So \(Y'\phi^{-1}\subseteq B\) implying that $Y'\subseteq B\phi$. Thus  \(Y'\subseteq B\phi \cap Y\).
    But by the choice of \(\phi\), \(B\phi \cap Y \subseteq F\).
    So \(Y'\subseteq F\cap Y\) and since $\phi$ fixes $F$ pointwise, $Y'\phi^{-1} = Y'\subseteq F\cap Y$.
    This is a contradiction since, as was mentioned before, $Y'\phi^{-1}\in\mathfrak{Y}'\theta$ and \(Y\) was chosen to satisfy the property given in \cref{sub-sub-sub}. 
    \end{proof}

    \begin{corollary}\label{finite_done}
        If \(\mathcal{F}\) is good, then it can be defined by a waning function.
    \end{corollary}
    \begin{proof}
        By \cref{wany}(1), \(\mathcal{F}\) has a filter base $\mathscr{B}$ consisting of wany sets. By \cref{bad_don't_exist}, every set $B\in \mathscr{B}$ is not bad, and so there exists a finitely wany set $A_{B}\in \mathcal{F}$ such that $A_B\subseteq B$. Hence the collection $\{A_{B}: B \in \mathscr{B}\}$ of finitely wany sets is a filter base for $\mathcal{F}$ as well.
        Thus by \cref{wany}(2), \(\mathcal{F}\) can be defined by a waning function.
     \end{proof}

\begin{definition}\label{de-wan-func}
    Let \(\mathcal{T}\) be a second countable semigroup topology on \(I_\N\) containing \(\Itwo\), with filter data \((\mathcal{F}_{\Tau,i})_{i\in \omega}\). By \cref{de-data}, every $\mathcal{F}_{\Tau,i}$ is good, and so, by Corollary \ref{finite_done}, there exists a waning function $f_i: \omega + 1 \to \omega + 1$ such that $\mathcal{F}_{\Tau,i}$ is the set of all neighbourhoods of $\varnothing$ in the topology $\mathcal{T}_{f_i}$.

    We define a function \(f_\mathcal{T}:(\omega +1) \to (\omega +1)\) by 
\begin{enumerate}
    \item \((i)f_{\mathcal{T}}= (0)f_i\) for all \(i\in \omega\);
    \item \((\omega)f_{\mathcal{T}}= \min\{(0)f_{\mathcal{T}},(1)f_{\mathcal{T}}, \ldots\}\).
\end{enumerate}
\end{definition}

One might think that $f_{\mathcal{T}}$ depends on the choice of the waning functions $f_i$ in \cref{de-wan-func}. However, if  \(f, g:\omega+1 \to \omega +1\) are waning functions, then the topologies $\mathcal{T}_{f}$ and $\mathcal{T}_g$ have the same neighbourhoods of \(\varnothing\) if and only if \((0)f=(0)g\) (by \cref{cor-nbhds-zero}). Hence if $f_i$ and $g_i$ are waning functions such that $\mathcal{F}_i$ is the set of neighbourhoods of $\varnothing$ in $\mathcal{T}_{f_i}$ and $\mathcal{T}_{g_i}$, then $(0)f_i = (0)g_i$. It follows that $f_{\mathcal{T}}$ is independent of the choice of the $f_i$ in \cref{de-wan-func}.

\begin{theorem}\label{Thm_peaking}
    If \(\mathcal{T}\) is a second countable semigroup topology on \(I_\N\) containing \(\Itwo\), 
    then \(f_{\mathcal{T}}\) is a waning function.
\end{theorem}
\begin{proof}
If $(\mathcal{F}_i)_{i\in\omega}$ is the filter data of $\mathcal{T}$, then by \cref{increasing_neighbourhoods}, $\mathcal{F}_m\subseteq \mathcal{F}_n$ for all $m\leq n$. So, by \cref{cor-nbhds-zero}, if $f_m$ and $f_n$ are waning functions corresponding to $\mathcal{F}_m$ and $\mathcal{F}_n$ (as in \cref{de-wan-func}), then 
$(0)f_m \geq (0)f_n$. In other words, $f_{\mathcal{T}}$ 
 is non-increasing. 

It therefore suffices to show that if \(i\in \N\) and \((i)f_{\mathcal{T}}\in \mathbb{N} \setminus \{0\}\), then \((i+1)f_{\mathcal{T}}<(i)f_\mathcal{T}\).

 \begin{claim}\label{neighbouthood+check}
  If \(g\in I_\N^{<\infty}\) and \(N\subseteq I_\N\), then \(N\) is a neighbourhood of \(g\) with respect to \(\Tau\) if and only if \(N\) contains the sets
  \[\makeset{h\in I_\N}{\(h{\restriction}_r= g{\restriction}_r\) and \(|\im(h)\cap (r\setminus \im(g))|\leq (|g|)f_\mathcal{T}\)}\]
  for all but finitely many \(r\in \N\).
 \end{claim}
 \begin{proof}
 Suppose that  \(g\in I_{\N}^{<\infty}\). 
Then by \cref{good_filter_def} and \cref{filter_independant_of_size},
\[\mathcal{F}_{|g|}=
\mathcal{F}_{\mathcal{T}, g}
= 
\{ (N\cap {\uparrow}g)d_{g} :N\text{ is a neighbourhood of } g \text{ with respect to }\Tau\}.
\] Hence $N$ is a neighbourhood of $g$ if and only if $(N\cap {\uparrow}g)d_{g}\in\mathcal{F}_{|g|}$. 
But $\mathcal{F}_{|g|}$ is the set of all neighbourhoods of $\varnothing$ in $\mathcal{T}_{f_{|g|}}$ for some waning function $f_{|g|}$ (as in \cref{de-wan-func}). 
By the definition of $f_{\Tau}$, \((|g|)f_{\mathcal{T}} = (0)f_{|g|}\), and so \cref{neighbourhood_basis} implies that 
 \(N\in \mathcal{F}_{|g|}\) if and only if \(N\) contains a set of the form
\[W_{f_{|g|}, \varnothing, r} = \makeset{h\in I_\N}{\(h{\restriction}_r= \varnothing\) and \(|\im(h)\cap r|\leq (0)f_{|g|}=(|g|)f_\mathcal{T}\)}\]
for some large enough \(r\in \N\). 
It follows that 
 \(N\) is a neighbourhood of \(g\) (with respect to \(\Tau\)) if and only if \((N\cap {\uparrow}g)d_g\) contains a set $W_{f_{|g|}, \varnothing, r}$ for large enough $r$.
 
We claim that the set
 \((N\cap {\uparrow}g)d_g\) contains a set $W_{f_{|g|}, \varnothing, r}$ for all but finitely many $r\in \N$
 if and only if
 \(N\) contains a set of the form
\[\makeset{h\in I_\N}{\(h{\restriction}_r= g{\restriction}_r\) and \(|\im(h)\cap (r\setminus \im(g))|\leq (|g|)f_\mathcal{T}\)}\]
for all but finitely many \(r\in \N\), as we can check that if \(r>|g|\) then
     \[\makeset{h\in I_\N}{\(h{\restriction}_r= g{\restriction}_r\) and \(|\im(h)\cap (r\setminus \im(g))|\leq (|g|)f_\mathcal{T}\)}=W_{f_{|g|, \varnothing, r-|g|}}d_g^{-1}.\]
Hence    
      \[\big(\makeset{h\in I_\N}{\(h{\restriction}_r= g{\restriction}_r\) and \(|\im(h)\cap (r\setminus \im(g))|\leq (|g|)f_\mathcal{T}\)}\big )d_g=W_{f_{|g|, \varnothing, r-|g|}}.\]
 \end{proof}
 

By \cref{neighbouthood+check}, the set 
\[N:=\makeset{f\in I_\N}{\(\id_i\subseteq f\) and \(|\im(f)\cap \{i, i+1, \ldots, i+ (i)f\}|\leq (i)f_\mathcal{T}\)}\]
is a neighbourhood of \(\id_i\) with respect to \(\Tau\). Let \(U\in \Tau\) be such that \(\id_i\in U\subseteq N\). 
Since \(U\) is a neighbourhood of \(\id_i\) and \((i)f_{\Tau}\neq 0\), we can find \(g\in U\) such that \(\im(g)=i+1\). As \(U\) is open, it is also a neighbourhood of \(g\). So \(N\) is a neighbourhood of \(g\).

It follows that
\begin{align*}
    N'&:=\makeset{f\in I_\N}{\(g\subseteq f\) and \(|\im(f)\cap \{i, i+1, \ldots, i+ (i)f_\mathcal{T}\}|\leq (i)f_\mathcal{T}\)}\\
    &= \makeset{f\in I_\N}{\(g\subseteq f\) and \(|\im(f)\cap \{ i+1, \ldots, i+ (i)f_\mathcal{T}\}|\leq (i)f_\mathcal{T}-1\)}
\end{align*}
is a neighbourhood of \(g\) with respect to \(\Tau\). Since \(|g|=i+1\), it follows that \((i+1)f_\mathcal{T}\leq (i)f_\mathcal{T}-1\) as required.
\end{proof}

\begin{theorem}
    If \(\mathcal{T}\) is a second countable semigroup topology on \(I_\N\) containing \(\Itwo\), then \(\mathcal{T}=\mathcal{T}_{f_{\mathcal{T}}}\). In particular every second countable, semigroup topology on \(I_\N\) containing \(\Itwo\) can be defined by a waning function.
\end{theorem}

\begin{proof}
By the definition of \(f_{\mathcal{T}}\), \cref{neighbourhood_basis}, and \cref{neighbouthood+check}, we know that for all \(g\in I_{\N}^{<\infty}\), \(g\) has the same neighbourhoods with respect to \(\mathcal{T}\) as \(\mathcal{T}_{f_{\mathcal{T}}}\) .
Thus we need only consider the neighbourhoods of the infinite elements of \(I_{\N}\).

We split the proof into two cases.
\begin{claim}
    If \((\omega)f_{\mathcal{T}}= 0\), then \(\mathcal{T}=\mathcal{T}_{f_{\mathcal{T}}}\).
\end{claim}
\begin{proof}
Let \(N\in\N\) be such that \((N)f_{\mathcal{T}}= 0\). Let \(a\in \N\) be arbitrary.
 We show that the set 
 \[\makeset{f\in I_\N}{\(N\leq |f|\) and \(a\not\in \im(f)\)}\]
  is open with respect to \(\Tau\).
  Note that this is sufficient as the \(\Ifour\) neighbourhoods of any infinite element of \(I_\N\) are generated by these sets and sets from \(\Itwo\).

  Let \(F_{N,a}:=\makeset{f\in I_\N}{\(|f|=N\) and \(a\notin \im(f)\)}\). Observe that
  
  \[\makeset{f\in I_\N}{\(N\leq |f|\) and \(a\not\in \im(f)\)}=\bigcup_{g\in F_{N, a}} \makeset{f\in I_\N}{\(g\subseteq f\) and \(a\not\in \im(f)\)}.\] 
  Thus we need only show that the sets in the above union are open.

  Endow the set \(\makeset{f\in I_\N}{\(N\leq|f|<\infty\)}\) with the subspace topology inherited from \((I_\N, \Ifour)\) and fix any \(g\in F_{N, a}\). 
  Define \(\phi:\makeset{f\in I_\N}{\(g\subseteq f\)} \to \makeset{f\in I_\N}{\(N\leq|f|<\infty\)}\) by 
  \[(f)\phi= f\circ \operatorname{id}_{\im(g)\cup \{a\}}.\]
  Note that \(\phi\) is continuous. It follows that 
  \[\big(\makeset{f\in I_\N}{\(a\notin \im(f)\)}\big)\phi^{-1}=\makeset{f\in I_\N}{\(g\subseteq f\) and \(a\not\in \im(f)\)}\]
  is open in \(\makeset{f\in I_\N}{\(g\subseteq f\)}\), which is also open in \(\Tau\supseteq \Itwo\) so the result follows.
\end{proof}
  \begin{claim}
      If \((\omega)f_{\Tau}=\omega\), then \(\mathcal{T}=\mathcal{T}_{f_{\mathcal{T}}} = \Itwo\).
  \end{claim}
  \begin{proof}
       
    Let \(W\in \Tau\), and \(g\in W\) be infinite. 
    We need only show that \(W\) is an \(\Itwo\) neighbourhood of \(g\). Consider the product \(\id_\N g=g\). As $\Itwo\subseteq\Tau\subseteq \Ifour$ and $\mathcal B:=\{{\uparrow}\id_n:n\in\N\}$ is an open neighborhood base of $\id_\N$ in both $\Itwo$ and $\Ifour$, we get the family $\mathcal B$ also forms a base at $\id_\N$ in $\Tau$.
    Since $\Tau$ is a semigroup topology, there is \(U\in \Tau\) and \(n\in \N\) such that
    \[({\uparrow} \id_n )U\subseteq W.\]
   Since $U\in\Ifour$, there exists finite \(g'\subseteq g\) such that \(g{\restriction}_{n}=g'{\restriction}_n\) and \(g'\in U\).
    As \(g'\) is finite and $(|g'|)f_\Tau=\w$, it follows that \(U\) is a \(\Itwo\) neighbourhood of \(g'\). So there is some finite \(k\in \N\) such that 
    \[U_{k,g'}:=\makeset{f\in I_\N}{\(f{\restriction}_{k}=g'{\restriction}_k\)}\subseteq U.\]
    In particular
    \[({\uparrow} \id_n )U_{k, g'}\subseteq W.\]
    Let \(m:= \max\{\{k,n\}\cup \dom(g')\cup \im(g')\}+1\). We have
    \[({\uparrow} \id_n )U_{m, g'}\subseteq W.\]
    It therefore suffices to show that \(U_{n, g'}\subseteq ({\uparrow} \id_n )U_{m, g'}\), as $U_{n, g'}$ is open in \(\Itwo\) and contains \(g\).
    Let \(h\in U_{n, g'}\) be arbitrary. Put \(h_1=g{\restriction}_n=g'{\restriction}_n=h{\restriction}_n\) and \(h_2=h\setminus h_1\). We have that 
    \[h=h_1 \cup h_2.\]
    Let \(l:\im(h_2) \to \N\setminus m\) be an injective function.

    Note that 
    \[\varnothing=\dom(\id_n)\cap \dom(h_2l)=\im(\id_n)\cap \im(h_2l)=\dom(l^{-1})\cap \dom(g')=\im(l^{-1})\cap \im(g').\]
    It follows that
    \[h=h_1 \cup h_2= (\id_n \circ g')\cup (h_2l\circ l^{-1})=(\id_n \cup h_2l) \circ (g'\cup l^{-1})\in ({\uparrow} \id_n)U_{m, g'}.\]
    Hence  \(U_{n, g'}\subseteq ({\uparrow} \id_n )U_{m, g'}\).
  \end{proof}
  The two claims above complete the proof of the theorem.
\end{proof}

\section{Proofs of the remaining main results}\label{section-corollaries}
In this section, we prove Theorems~\ref{newone}, \ref{newversion},  and \ref{corollary-order}. We start with Theorem~\ref{newone}, which states that the symmetric inverse monoid $I_\N$ endowed with any \(T_1\) second countable semigroup topology is homeomorphic to the Baire space $\N^\N$.

\begin{proof}[Proof of Theorem~\ref{newone}]
By \cref{Polish}, the spaces $(I_\N,\Tau_{f})$ 
and $(I_\N,\Tau_f ^{-1})$ are Polish and possess a base consisting of clopen sets for every waning function $f$. 
By the Alexandrov-Urysohn Theorem (see, for example,~\cite[Theorem 7.7]{Kechris}), a non-empty Polish space $X$ possessing a base consisting of clopen sets is homeomorphic to the Baire space $\N^\N$ if and only if each compact subset of $X$ has empty interior. Recall that for each waning function $f$ the spaces $(I_\N,\Tau_f)$ and  $(I_\N,\Tau_f^{-1})$ are homeomorphic. By \cref{main-theorem}, it is enough to show that $(I_\N,\Tau_f)$ is homeomorphic to the Baire space for each waning function $f$. 
Since all the topologies \(\Tau_f\) fall between \(\Itwo\) and \(\Ifour\), it suffices to check that the sets which are compact with respect to \(\Itwo\) have empty interior with respect to \(\Ifour\). If this was not the case, then there would be a basic open set $U$ in \(\Ifour\) whose closure in \((I_\N,\Itwo)\) is compact. 
Then there exist $n\in\N$, a finite subset $X$ of $N$ and a finite partial bijection $h: n\to \N\setminus X$ such that 
$$U=\makeset{f\in I_{\mathbb{N}}}{\(f{\restriction}_n=h\) \hbox{ and } \(\im(f)\cap X=\varnothing\)}.$$ 
It is straightforward to check that the set \(U\) is closed with respect to \(\Itwo\) and the open cover \[\mathcal U=\makeset{\makeset{f\in I_{\mathbb{N}}}{\((n)f= m\)}}{\(m\in \N\)} \cup\left\{\makeset{f\in I_{\mathbb{N}}}{\(n\notin \dom(f)\)}\right\}\] has no finite subcover. 
\end{proof}

Recall that $\mathfrak{P}$ denotes the poset of Polish semigroup topologies on $I_{\N}$, with respect to containment, and
$\mathfrak{W}$ the set of all waning functions equipped with the partial order $\preceq$, where $f\preceq g$ if and only if $(n)f\geq (n)g$ for all $n\in\omega+1$.
We proceed with~\cref{newversion}, which states that

(i) the function $f \mapsto \mathcal{T}_f$ is an order-isomorphism from $\mathfrak{W}$
to the subposet of $\mathfrak{P}$ consisting of Polish semigroup topologies which contain $\Itwo$ and are contained in $\Ifour$;

(ii) the function $f \mapsto \mathcal{T}_f ^ {-1}$ is an order-isomorphism from
$\mathfrak{W}$ to the subposet of $\mathfrak{P}$ of Polish semigroup topologies lying between $\Ithree$ and $\Ifour$.

\begin{proof}[Proof of Theorem~\ref{newversion}]

\vspace{\baselineskip}\noindent
(i) \cref{lem-order-preserving} and \cref{newcol} imply that the function $f \mapsto \mathcal{T}_f$ is injective and order preserving. By~\cref{main-theorem}, this function is surjective and, as such, an order-isomorphism from
$\mathfrak{W}$ to the subposet of $\mathfrak{P}$ of Polish semigroup topologies lying between $\Itwo$ and $\Ifour$. 

\vspace{\baselineskip}\noindent
(ii) Observe that $\Itwo^{-1}=\Ithree$, $\Ifour^{-1}=\Ifour$ and $\Tau_f^{-1}\subseteq \Tau_g^{-1}$ if and only if $\Tau_f\subseteq \Tau_g$ for any waning functions $f$ and $g$. At this point the proof of item (ii) is dual to the one of (i).  
\end{proof}

Finally, we prove \cref{corollary-order}, which states that the poset $\mathfrak{P}$ contains the following:
  \begin{enumerate}[\rm (a)]
  \item infinite descending linear orders;
  \item only finite ascending linear orders; and 
  \item every finite partial order.
  \end{enumerate}

\begin{proof}[Proof of~\cref{corollary-order}.]
By \cref{newversion}, it suffices to prove that $\mathfrak{W}$ satisfies $(a), (b)$, and $(c)$. 

\vspace{\baselineskip}
\noindent
\((a)\) For every $n \in \N$, define $f_n: \omega +1 \rightarrow \omega+1$ by $(0)f_n=n$ and $(i)f_n=0$ for all $i>0$. Then each $f_n$ is waning and $m\geq n$ implies that $(i)f_m\geq (i)f_n$ for all $i\in \N$. Thus $(f_n)_{n\in \N}$ is an infinite descending chain in $\mathfrak{W}$.

\vspace{\baselineskip}
\noindent
\((b)\)
Suppose for a contradiction that \((f_i)_{i\in \N}\) is an infinite ascending linear order with respect to \(\mathfrak{W}\). In other words, with respect to the pointwise order, we have \(f_0> f_1>f_2 \ldots\).
Note that as \(f_1<f_0\) with respect to the pointwise order,  there is some \(n\in \mathbb{N}\) such that \((n)f_1\neq \omega\). 

Let \(k=\min \makeset{j\in \N}{\(\exists p\in \mathbb{N}\ (j)f_p\neq \omega\)}\). Let \(p\in \N\) be such that \((k)f_p<\omega\).
Then for all \(i\in \N\) and \(j<k\), we have \((j)f_i=\omega\).

It follows from \(f_p\) being a waning function that \((k+ (k)f_p)f_p = 0\), and hence there is \(m\in \mathbb{N}\) for which \((i)f_p=0\) for all \(i\geq m\).
Observe that there are at most \(((k)f_p+1)((k+1)f_p+1)\ldots ((m)f_p+1)<\infty\) functions in our sequence bounded above pointwise by \(f_p\), which yields a contradiction.

\vspace{\baselineskip}
\noindent
\((c)\) Let \((X, \leq)\) be a finite poset. 
First note that the map 
\(x\mapsto {\downarrow} x\)
is an embedding of \((X, \leq)\) into \((\mathcal{P}(X), \subseteq)\). Moreover the map from \(\mathcal{P}(X)\) to \(\{0, 1\}^X\) which sends each set \(S\) to its characteristic function (\(x\mapsto 1 \iff x\in S\)) is an order isomorphism with respect to the coordinate-wise order on \(\{0, 1\}^X\). So it is sufficient to embed  \(\{0, 1\}^X\) into $\mathfrak{W}$. Assume without loss of generality that \(X = n\) and note that for example the set
\[\makeset{f\in \mathfrak{W}}{\((n)f=0\) and for all \(i\in n\), we have \((i)f\in \{3(n-i) +1, 3(n-i) +2\}\)}\]
is order isomorphic to \(\{0, 1\}^n\), as required.
\end{proof}

\end{document}